\documentclass[11pt]{article}
\usepackage[letterpaper, hmargin=1.1in, top=1in, bottom=1.2in, footskip=0.6in]{geometry}
\usepackage{titlesec}
\titleformat{\subsection}[runin]{\normalfont\bfseries}{\thesubsection.}{.5em}{}[.]\titlespacing{\subsection}{0pt}{2ex plus .1ex minus .2ex}{.8em}
\titleformat{\subsubsection}[runin]{\normalfont\itshape}{\thesubsubsection.}{.3em}{}[.]\titlespacing{\subsubsection}{0pt}{1ex plus .1ex minus .2ex}{.5em}
\usepackage[labelfont=sc,font=small,labelsep=period]{caption}
\usepackage{amsmath} 
\usepackage{amssymb}
\usepackage{amsfonts}
\usepackage{latexsym}
\usepackage{amsthm}
\usepackage{amsxtra}
\usepackage{amscd}
\usepackage{bbm}
\usepackage{mathrsfs}
\usepackage{bm}
\usepackage{xcolor}
\usepackage{stmaryrd}

\usepackage[nottoc,notlof,notlot]{tocbibind}
\usepackage{cite}
\usepackage{graphicx, color}
\definecolor{darkred}{rgb}{0.9,0,0.3}
\definecolor{darkblue}{rgb}{0,0.3,0.9}

\usepackage{ifthen}
\def\comment#1{\ifthenelse{\isodd{\value{page}}}{\marginpar{\raggedright\scriptsize{\textcolor{darkred}{#1}}}}{\marginpar{\raggedleft\scriptsize{\textcolor{darkred}{#1}}}}}  
\definecolor{vdarkred}{rgb}{0.6,0,0.2}
\definecolor{vdarkblue}{rgb}{0,0.2,0.6}
\usepackage[pdftex, colorlinks, linkcolor=vdarkblue,citecolor=vdarkred]{hyperref}
\numberwithin{equation}{section}
\numberwithin{figure}{section}

\theoremstyle{plain} 
\newtheorem{theorem}{Theorem}[section]
\newtheorem*{theorem*}{Theorem}
\newtheorem{lemma}[theorem]{Lemma}
\newtheorem*{lemma*}{Lemma}

\newtheorem*{corollary*}{Corollary}
\newtheorem{proposition}[theorem]{Proposition}
\newtheorem*{proposition*}{Proposition}
\newtheorem{definition}[theorem]{Definition}
\newtheorem*{definition*}{Definition}

\newtheorem*{conjecture*}{Conjecture}

\theoremstyle{definition} 

\newtheorem*{example*}{Example}
\newtheorem{remark}[theorem]{Remark}
\newtheorem*{remark*}{Remark}

\renewcommand{\leq}{\leqslant}
\renewcommand{\geq}{\geqslant}
\renewcommand{\epsilon}{\varepsilon}

\begin{document}

\begin{center}

\begin{huge}{Proof of the Weak Local Law for Wigner Matrices using Resolvent Expansions }
\end{huge}

\vspace{10mm}
Vlad Margarint\\

\end{center}

\vspace{5mm}

\begin{abstract}
The aim of this paper is to provide a novel proof for the Local Semicircle Law for the Wigner ensemble. The core of the proof is the intensive use of the algebraic structure that arises, i.e. resolvent expansions and resolvent identities. On the analytic side, concentration of measure results and high probability bounds are used. The conclusion is obtained using a bootstrapping argument that provides information about the change of the bounds from large to small scales. This approach leads to a new and shorter proof of the Weak Local Law for Wigner Matrices, that exploits heavily the algebraic structure appearing in the setting.
\end{abstract}

\vspace{20mm}

\begin{Large}

\textbf{Introduction}
\vspace{5mm}

\end{Large}
Random matrices are fundamental objects for the modeling of complex systems. A basic example is the Wigner ensemble. The Wigner matrices $H$ are $N \times N$ symmetric or Hermitian matrices whose entries are random variables that are independent up to the constraint imposed by the symmetry $H=H^*\,.$ Wigner made the observation that the distribution of the distances between  consecutive eigenvalues (that is called the gap distribution) follows an universal pattern. Furthermore, he predicted that the universality is not restricted to the Wigner ensemble, but it should hold for any system of sufficient large complexity, described by a large Hamiltonian. The Wigner-Dyson-Gaudin-Mehta conjecture (\cite{erdHos2010bulk}, \cite{erdHos2011universality}, \cite{erdHos2012local}) states that the symmetry class of the matrix is the one that dictates the gap distribution, and not the details of the distribution of the matrix entries. The Local Law for Wigner Matrices is not a new result. It was originally proved in the following series of works \cite{erdHos2013local}, \cite{erdHos2009local}, \cite{erdHos2012rigidity}, \cite{tao2013random}, \cite{tao2011random} and was used to solve the Wigner-Dyson-Mehta conjecture.
 Statements analogous to the Weak Local Law for Wigner Matrices have been proved for many models, including hermitian matrices with i.i.d entries \cite{erdHos2009local}, \cite{erdHos2012rigidity}, \cite{tao2013random}, general independent entries \cite{ajanki2017universality}, i.i.d. entries and general expectations \cite{he2017isotropic}, \cite{lee2013local}, correlated entries \cite{ajanki2016local}, \cite{ajanki2016stability}, \cite{che2017universality}, \cite{erdHos2017random}, $d$-regular graphs \cite{bauerschmidt2017local}, sample covariance matrices \cite{bloemendal2014isotropic}, \cite{knowles2017anisotropic}, \cite{alt2017local} as well as non-hermitian random matrices with i.i.d. entries \cite{bourgade2014local}, general independent entries \cite{alt2018local}, products of matrices with i.i.d entries \cite{nemish2017local} and $\beta$-ensembles \cite{bourgade2014universality}.\\
 In this paper, we focus on the Wigner Hermitian case. The strategy of proof is based on expanding on individual entries and using the resolvent expansion and resolvent identity formula. The use of expanding in individual entries of the random matrix was also used in the proof in \cite{erdos2010universality}. The resolvent expansion is very similar in spirit with the cumulant expansion technique used in \cite{he2017mesoscopic}, \cite{khorunzhy1996asymptotic} and \cite{lee2015edge} for more general models. However, the current work to the knowledge of the author, is the only paper that uses the resolvent expansion technique in the context of Wigner matrices. The advantage of working with Resolvent identities is the simplification of the proof of the Weak Local Law in the case of Wigner Matrices, by using extensively the algebraic identities that appear in the setting. Also, the use of this method provides a short proof of the Weak Local Law in this case. The paper is kept at the current length only for the completeness of the exposure, but one can simplify some parts of the proof further. The main advantage of the method is that it provides a simpler and shorter proof, compared with the other works on the topic. \\
 
The paper is divided into four sections. The first section contains the main definitions and notations that are used throughout the paper. In the second section, we use the resolvent expansion  to obtain a concentration result, namely the Concentration of Measure Lemma. The main ingredient for proving the result is the use of a generalized version of Efron-Stein inequality. In the third section, the resolvent expansion along with the Concentration of Measure Lemma are used to provide high probability bounds for the quantity $1+sz+s^2$, where $s=\frac{1}{N}\sum_{i=1}^NG_{ii}\,,$ with $(G_{ij})_{i,j \in [1,N]}$ being the resolvent matrix of the Wigner matrix $(H_{ij})_{i, j \in [1, N]}$. The last section is divided into two parts. The first part is a stability analysis argument done rather straightforward that shows that the Stieltjes transform of the Semicircle Law is close to the average $s \;:=\; \frac{1}{N}\sum_{i=1}^NG_{ii}\,.$ The last step is a bootstrapping argument that involves all the results discovered in the previous steps. The main use of this argument is to track the information about the bounds from the scale $\eta$ to the scale $\eta/N^{\delta}\,,$ for $\delta >0\,.$

The papers \cite{he2017mesoscopic}, \cite{he2017isotropic} develop a more general set of ideas that do not cover this result directly. Given this sequence of papers, we believe that this result consolidates this direction of research.
 
\vspace{4mm}
 
\textbf{Acknowledgements.} This result is obtained in ETH Z\"urich under the supervision of Prof. Dr. Antti Knowles. The author is grateful to Prof. Dr. Antti Knowles for the careful guiding into understanding the problem.

\vspace{10mm}

\section{Preliminaries, Notations and Main Theorem}

Let $H=H^*$ be a $N \times N$ random Hermitian matrix with eigenvalues $ \lambda_1 \geq \lambda_2 \geq \lambda_3 \geq\cdots \geq \lambda_N\,.$ We normalize the matrix $H$ such that its eigenvalues are of order one.

Since most quantities that we are interested in depend on $N$ we shall almost always omit the explicit argument $N$ from our notations. Hence, every quantity that is not explicitly a constant is a sequence indexed by $N \in \mathbb{N}\,.$

In order to study the eigenvalues of the matrix $H$ we define the $resolvent$  by 
\begin{align*}
G(z)\; := \; (H-zI_n)^{-1},
\end{align*}
where $z \in \mathbb{C} \setminus \{\lambda_1,\lambda_2, \lambda_3,\ldots,\lambda_N \}$ is the $\textit{spectral parameter}\,.$ We identify in the following the matrix $zI_n$ with $z$\,. From now on, we use the notation 
\begin{align*}
z=E+i\eta
\end{align*}
for the real and imaginary part of $z$\,. We always assume that $\eta>0$\,, i.e. $z$ lies in the upper half-plane.

Let us consider $ \tilde{H} = H + \Delta $ to be a perturbation of $H$. Then the resolvent of $\tilde{H}$ is given by $\tilde{G}(z)\;:=\;(\tilde{H}-z)^{-1}.$ By iterating the resolvent identity $\tilde{G}\;=\;G-G\Delta\tilde{G}$\,, we get, as in \cite{benaych2016lectures},  the resolvent expansion 
\begin{equation}\label{eq4}
\tilde{G}(z) \;=\; \sum\limits_{i=0}^{N-1}G(z)(-\Delta G(z))^{i} + \tilde{G}(z)(-\Delta G(z))^{N}.
\end{equation}

An important property that the resolvent satisfies is the $\textit{Ward identity}$  given by
\begin{equation}\label{eq5}
\sum\limits_{j=1}^N|G_{ij}|^2=\frac{1}{\eta}\operatorname{Im}G_{ii}\,.
\end{equation}

From now on, we consider the case when $H$ is a matrix as in the following definition.
\begin{definition}\label{defWigner}
A Wigner matrix is a $N \times N$ random Hermitian matrix $H \; =\;H^*$ whose entries $H_{ij}$ satisfy the following conditions.
\begin{enumerate}
\item The upper-triangular entries $(H_{ij}:1 \leq i \leq j \leq N)$ are independent. 
\item For all $i, j,$ we have $\mathbb{E}(H_{ij})=0$ and $\mathbb{E}|H_{ij}|^{2}=N^{-1}(1+O(\delta_{ij}))$.
\item The random variables $\sqrt{N}H_{ij}$ are bounded in any $L^p$ space, uniformly in $N, i , j\,.$
\end{enumerate}
\end{definition}

The conditions $(i)$ and $(ii)$ are the usual ones. Condition $(iii)$ states that for each $p \in \mathbb{N}$  there exists a constant $C_p$ such that $|| \sqrt{N}H_{ij}||_p \leq C_p$ for all $N, i, j,$ where we used the notation
$$||X||_p \;:=\;(\mathbb{E}|X|^p)^{1/p}\,. $$

We define as in \cite{benaych2016lectures} the $\textit{Semicircle Law}$ 
\begin{equation}
\rho(dx)\;:=\;\frac{1}{2\pi}\sqrt{(4-x^2)_{+}}dx\,,
\end{equation} 
and its Stieltjes transform 
\begin{equation}
m(z)\;:=\; \int{\frac{\rho(dx)}{x-z}}\,.
\end{equation}

In order to state the Local Semicircle Law we introduce, as in \cite{benaych2016lectures}\,, the notion of high probability bounds. 
\begin{definition}[Stochastic domination]
Let $U^{(N)}$ a possibly $N$-dependent parameter set.\\
Let $X=\left(X^{(N)}(u): N \in \mathbb{N}, u \in U^{(N)}\right)$ and $Y=\left(Y^{(N)}(u) : N \in \mathbb{N}, u \in U^{(N)} \right) $ be two families of nonnegative random variables. We say that $X$ is stochastically dominated by $Y$, uniformly in $u$, if for all $\epsilon > 0$ and large $D>0$ we have 
\begin{equation}
\sup_{u \in U^{(N)}}\mathbb{P}\left[ X^{(N)}(u) > N^{\epsilon}Y^{(N)}(u)\right] \leq N^{-D}
\end{equation}
for large $N \geq N_{0}(\epsilon, D)\,.$\\

The stochastic domination is always uniform in all parameters that are not explicitly fixed.\\

If $X$ is stochastically dominated by $Y$, uniformly in $u$, we use the notation $X \prec Y$ or $X=O_{\prec}(Y)$\,. Moreover, if for some complex family $X$ we have that $|X| \prec Y\,,$ we also write $X=O_{\prec}(Y)\,.$\\

We say that an event $\Xi \equiv \Xi^{(N)}$ holds with high probability if $1-\mathbf{1}(\Xi) \prec 0$\,, i.e. if for any $D>0$ there is $N_0(D)$ such that for all $N \geq N_0(D)$ we have that $\mathbb{P}(\Xi^{(N)})\geq 1-N^{-D}$\,.
\end{definition}
\noindent
The notion of $\textit{ stochastic domination}$ provides a simple way of making precise statements of the form $"X\textit{is bounded with high probability by Y up to small powers of} \hspace{1.75mm} N^{\epsilon}".$ Note that this notation implicitly means that $\prec$ is uniform in the indices $i, j \in \llbracket 1, N \rrbracket\,.$

For fixed $\gamma \geq 0$ we consider the domain 
$$\bold{S}\equiv \bold{S}_{N}(\gamma)\;:=\;\{z= E+i\eta :  -N \leq E \leq N, \hspace{2mm}  N^{-1+\gamma} \leq \eta \leq N  \}\,. $$
The parameter $\gamma$ should be considered forever fixed. Throughout this paper,  all the estimates depend on $\gamma$ but we do not indicate or track this dependency in the equations or identities that we study.

For proving the main results, we use the spectral parameters
\begin{align}\label{equiv5.1}
z_0=E+i\eta_0\,, \hspace{4mm}
z=E+i\eta\,, \hspace{4mm}
N^{-1+\gamma} \leq \eta_0 \leq \eta \leq N\,.
\end{align}
We also use the function $F_z : [0,1] \to \mathbb{R}$ defined by
\begin{equation}
F_z(r)\;:=\; \left[\left(1+\frac{1}{\sqrt{|z^2-4|}}\right)r \right]\wedge \sqrt{r} \,.
\end{equation}

Our main result is a new way of proving the following Theorem  using intensively the algebraic structure of the the resolvent expansions and resolvent identities\,.
\begin{theorem}
[Local Semicircle Law for Wigner Matrices]\label{theoremlocal}
Let $H$ be a Wigner matrix and let 
\begin{equation}
\psi(z) \;:=\; \frac{1}{\sqrt{N\eta}} 
\end{equation} 
be a deterministic error parameter.\\
Then, we have that 
\begin{align}
\max_{i \in \llbracket 1, \ldots, N \rrbracket}|G_{ii}(z)-m(z)|&\prec F_{z}(\psi(z))\,,\nonumber\\
\max_{i \neq j }|G_{ij}(z)| &\prec \psi(z) \,,
\end{align}
uniformly for all $ z \in \mathbb{C}_{+}\,,$ such that $\eta \geq N^{-1+\gamma}\,.$
\end{theorem}
\begin{remark}
Note that the rate of convergence at the edge of the spectrum is not optimal (see the works \cite{erdHos2012rigidity} and \cite{erdHos2013local}).
\end{remark}

We now start to introduce the notations that we use throughout the paper.

Let us consider the indices $i, j, k, l \in \llbracket 1, N \rrbracket\,.$ We introduce the following notation
\begin{equation}
H_{kl}^{(ij)}\;:=\; H_{kl}\bold{1}\left( \{i,j\}\neq \{k,l \}\right)\,.
\end{equation}
We further define
\begin{equation}
\Delta^{(ij)} \;:=\; H-H^{(ij)}\,.
\end{equation}
Hence, the matrix $H^{(ij)}$ is an $N \times N $ matrix with the same elements as $H$ except that $H^{(ij)}_{ij}=0$ and $H^{(ij)}_{ji}=0$\,. 
For $i,j \in \llbracket 1, N \rrbracket$ the resolvent of $H^{(ij)}$ is given by
\begin{equation}
G^{(ij)}(z) \;=\; (H^{(ij)}-z)^{-1}\,.
\end{equation}

We mention that for the proof of the main Theorem, we often omit the spectral parameter $z$ from our notation, bearing in mind that we are typically dealing with random sequences indexed by $N \in \mathbb{N}$ of functions of $z \in \bold{S}\,.$

Let us consider a particular perturbation that is useful throughout the paper. Let $i \in \llbracket 1, N \rrbracket$ be a fixed index. We consider for $H$ as in the definition \ref{defWigner}\,, the perturbation
\begin{equation}
H=H^{(1i)}+\Delta^{(1i)}\,.
\end{equation}
Varying $i \in \llbracket 1, N \rrbracket\,,$ we obtain the perturbation of $H$ in each index. In order to simplify the notations in the identities that we study, throughout the paper we use the notation $h_{i}\; :=\; H_{1i}\,, $ for all $i \in \llbracket 1, N \rrbracket$\,.

For $k, l \in \llbracket 1, N \rrbracket\,,$ we define for each $i \in \llbracket 1, N \rrbracket\,,$ the functions $G_{kl}^{i} : \mathbb{C} \to \mathbb{C}$ and $\Gamma^i : \mathbb{C} \to \mathbb{C}$ by
\begin{align*}
G_{kl}^i(x)\;&:=\;G_{kl}(h_1,\ldots,h_{i-1}, x, h_{i+1},\ldots,h_N)\,,\\ 
\Gamma^i(x)\;&:=\;\max_{k,l \in \llbracket 1, N \rrbracket}G^i_{kl}(h_1,\ldots,h_{i-1},x,h_{i+1},\ldots h_N)\,.
\end{align*}
From now on, we use the simplified notations $G^i_{kl}(\cdot) \equiv G^i_{kl} (h_1, \ldots,h_{i-1},\cdot, h_{i+1},\ldots h_N)\,,$ for $k,l,i$ $\in \llbracket 1, N \rrbracket$ and $\Gamma^i (\cdot) \equiv \Gamma^i (h_1, \ldots,h_{i-1}, \cdot, h_{i+1},\ldots, h_N)\,,$ for $i \in \llbracket 1, N \rrbracket\,.$

We use two types of notations when  performing resolvent expansions. The type of notation that we choose depends on the context. The scope of this is to emphasize  the proprieties of the resolvents that we use in each situation. In the following we give the connection between the two notations.
\begin{enumerate}
\item
For $k, l, i \in \llbracket 1, N \rrbracket\,,$ using the definition of $G_{kl}^i (x) $ we obtain that 
\begin{align}
|G_{kl}^i(h_i)|=|G_{kl}|\,.
\end{align}
\item
For $k, l, i \in \llbracket 1, N \rrbracket \,,$ using the definition of $H^{(ij)}$ and the definition of $G_{kl}^i(x)$ we obtain that 
\begin{align}
|G_{kl}^i(0) |= |G_{kl}^{(1i)}|\,.
\end{align}
\end{enumerate}

\section{Concentration of Measure}

In this section we prove a concentration result that is further used in the proof of the main Theorem. Before doing so, we prove two lemmas that provide useful bounds for  the entries of the resolvent. The main ingredients for proving the result of this section are a generalized version of the Efron-Stein inequality from \cite{boucheron2013concentration} and a truncation of the summations that appear with respect to two events. We start by proving the result for a fixed scale using the Chebyshev inequality, and after that we compute the equivalent bounds on smaller scales. Moreover, we first obtain estimates for fixed $k,l \in \llbracket 1, N \rrbracket\,,$ and after that, we let $k$ and $l$ vary in $\llbracket 1, N \rrbracket $ such that we obtain the same estimates for every entry of the resolvent.

We consider the matrix $\theta$ given by
\begin{align*}
\theta \;:=\;\left(H_{kl}\right)_{k,l\hspace{1mm} \in \llbracket2,N \rrbracket}\,.
\end{align*}
We also consider the probability space 
$\Omega=\{\theta, h_1,\ldots,h_N \}$ and we introduce the notation $$ \mathbb{E}_{1}\;:=\;\mathbb{E}(\cdot | \theta)$$ for the conditional expectation with respect to the elements of the matrix $ \theta$\,.

Let us define 
$$V_{\operatorname{Re}} \;:=\; \sum\limits_{i=1}^N(\operatorname{Re}G^i_{kl}(h_i)-\operatorname{Re}G^i_{kl}(0))^2\,,$$
and
$$V_{\operatorname{Im}} \;:=\; \sum\limits_{i=1}^N(\operatorname{Im}G^i_{kl}(h_i)-\operatorname{Im}G^i_{kl}(0))^2\,.$$

Using that $||f ||_{2q}^{2q} =||f_+||_{2q}^{2q}+||f_-||_{2q}^{2q} \hspace{4mm} \forall q \geq 1, \forall f \in \text{L}_{2q}$\,, we obtain that
\begin{align}
||\operatorname{Re}G_{kl}-\mathbb{E}_1\operatorname{Re}G_{kl}||_{2q}^{2q}=||(\operatorname{Re}G_{kl}-\mathbb{E}_1\operatorname{Re}{G_{kl}})_+||_{2q}^{2q}+||(\operatorname{Re}G_{kl}-\mathbb{E}_1\operatorname{Re}{G_{kl}})_-||_{2q}^{2q}\,.
\end{align}
Let $k_q \in \mathbb{R}\,,$ with $k_q <2\,.$ Using  the estimates in Theorem $15.5 $ in \cite{boucheron2013concentration}\,, we obtain that 
\begin{align}\label{equation1}
||(\operatorname{Re}G_{kl}-\mathbb{E}_1\operatorname{Re}G_{kl})_+||_{2q}^{2q}&+||(\operatorname{Re}G_{kl}-\mathbb{E}_1\operatorname{Re}G_{kl})_-||_{2q}^{2q}\nonumber\\
&\leq\sqrt{2k_q}||\sqrt{V_{\operatorname{Re}}^{+}}||_{2q}^{2q}+\sqrt{2k_q}||\sqrt{V_{\operatorname{Re}}^{-}}||_{2q}^{2q}\nonumber\\
&=\sqrt{2k_q}\mathbb{E}(|V^+_{\operatorname{Re}}|^q+|V^-_{\operatorname{Re}}|^q)\nonumber\\
&\leq 2\mathbb{E}|V^+_{\operatorname{Re}}+V_{\operatorname{Re}}^-|^{q}\nonumber\\
&=2||V_{\operatorname{Re}}||_{q}^{q}\,.
\end{align}
In the same manner, we obtain that
\begin{align}\label{fortheimaginary}
||\operatorname{Im}G_{kl}-\mathbb{E}_1\operatorname{Im}G_{kl}||_{2q}^{2q} \leq 
2||V_{\operatorname{Im}}||_{q}^{q}\,.
\end{align}

From now on, we fix $\delta \in (0, \delta_0)\,,$ where $\delta_0  \leq \left(0, \frac{\gamma}{3} \right).$
For the fixed value of $ \delta \in (0, \delta_0)\,, $  let $\epsilon_0$ be such that $ \epsilon_0 \leq \frac{3}{4}\left(\frac{1}{2}-\delta-\frac{\log 4}{\log N} \right) \,.$

The following result gives a bound for the entries of the Wigner matrices.
\begin{lemma}\label{2.1}
Let $H$ be a Wigner matrix. 
Then 
\begin{align*}
|H_{ij}| \prec N^{-1/2} \hspace{4mm} \forall i,j \in \llbracket 1, N \rrbracket\,.
\end{align*}
\end{lemma}
\begin{proof}
Let us fix $ \epsilon \in (0, \epsilon_0)\,.$
Using property $(iii)$ from the definition of the Wigner matrices we obtain that 
\begin{align*}
||H_{ij}||_p^p \leq \frac{C_p}{N^{p/2}}\,, \hspace{3mm} \forall i, j \in \llbracket 1, N \rrbracket \,.
\end{align*}
Using Chebyshev inequality, we obtain that
\begin{align*}
\mathbb{P}( |H_{ij}| \geq \frac{N^{\epsilon}}{\sqrt{N}} ) \leq \frac{N^{p/2}||H_{ij} ||_p^p}{N^{p\epsilon}}\leq \frac{C_p}{N^{p\epsilon}}\,.
\end{align*}
Given $\epsilon \in (0, \epsilon_0)$ we can choose $p$ large enough such that we can apply the definition of the stochastic domination and conclude the proof.
\end{proof}
Let $x \in \{0, h_1,\ldots,h_N\}\,.$ For fixed $ \epsilon \in (0, \epsilon_0)\,,$ we define the following events
\begin{align}\label{event}
\Xi \;&:=\; \{ \Gamma^i(x) \leq N^{\epsilon/3 + \delta}\,, i \in  \llbracket 1, N \rrbracket\}\,,\nonumber\\
\tilde{\Xi} \;&:=\; \{ |H_{ij}| \leq N^{\epsilon-1/2}\,, i,j \in \llbracket 1,N \rrbracket \}\,.
\end{align}

For proving the main result of this section, we need also the following lemmas.
\begin{lemma}\label{Gamma0}
Let $\epsilon \in (0, \epsilon_0)\,.$ For each $i \in \llbracket 1, N \rrbracket \,,$ we have that \\
1)On the events $\Xi$ and $\tilde{\Xi}$ it holds $$|G^i_{kl}(0)|\leq 2N^{\epsilon/3+\delta}\,,\hspace{3mm} \text{for all} \hspace{1mm} k, l \in \llbracket 1, N \rrbracket\,.$$
2)On the event $\Xi^c\,,$ it holds
$$|G^i_{kl}(0)|\leq N \,,\hspace{3mm} \text{for all} \hspace{2mm} k, l \in \llbracket 1, N \rrbracket \,.$$

\end{lemma}
\begin{proof}
Let us fix arbitrarily $\epsilon \in (0, \epsilon_0)$\,. Let us fix also $i \in \llbracket 1, N \rrbracket\,.$
\begin{enumerate}
\item  
Using the resolvent identity, we obtain that
\begin{align*}
|G^i_{kl}(0)\bold{1}(\Xi)\bold{1}(\tilde{\Xi})|&=|G^i_{kl}(h_i)\bold{1}(\Xi)\bold{1}(\tilde{\Xi})-G^i_{ki}(h_i)H_{i1}G^i_{1l}(0)\bold{1}(\Xi)\bold{1}(\tilde{\Xi})\\&-\bold{1}(i\neq 1)G^i_{k1}(h_i)H_{1i}G^i_{il}(0)\bold{1}(\Xi)\bold{1}(\tilde{\Xi})|\,.
\end{align*}
Using that $|H_{1i}|\leq N^{\epsilon-1/2}$ on $\tilde{\Xi}$ and that $\Gamma^i(h_i)\leq N^{\epsilon/3+\delta}$ on $\Xi$\,, we obtain
\begin{equation}
\Gamma^i(0) \leq \Gamma^i(h_i)+2\Gamma^i(h_i)N^{\epsilon-1/2}\Gamma^i(0)
\leq N^{\epsilon/3+\delta}+2N^{\epsilon/3+\delta}N^{\epsilon-1/2}\Gamma^i(0)\,.
\end{equation}
Since $\epsilon \in (0, \epsilon_0)\,,$ we obtain that $2N^{\epsilon/3+\delta}N^{\epsilon-1/2} \leq 1/2 \,,$
that gives the desired conclusion
\begin{equation}
|G^i_{kl}(0)| \leq \Gamma^i(0) \leq 2N^{\epsilon/3+ \delta}\,.
\end{equation} 
\item On the complementary event $\Xi^c\,,$  we use the brutal estimate $|G_{kl}^i(h_i)| \leq N,$ for all $ k,l \in \llbracket 1, N \rrbracket\,.$ Being the resolvent of a Hermitian matrix, the same bound applies to $|G_{kl}^i(0)|\,,$ i.e. $|G_{kl}^i(0)| \leq N$ on $\Xi^c\,.$
\end{enumerate}
Varying $i \in \llbracket 1, N \rrbracket$\,, we obtain the conclusion.
\end{proof}

\begin{lemma}\label{2.3}
Let $k, l, i \in \llbracket 1, N \rrbracket\,.$ On the events $\Xi$ and $\tilde{\Xi}\,,$ we have the estimate
$$|G_{kl}^i(0)| \leq 2|G_{kl}^i(h_i)|=2|G_{kl}|\,.$$
\end{lemma}
\begin{proof}
Let us fix arbitrarily $\epsilon \in (0, \epsilon_0)\,.$ 
Using the resolvent identity, we obtain that
\begin{align*}
|G^i_{kl}(0)\bold{1}(\Xi)\bold{1}(\tilde{\Xi})|&=|G^i_{kl}(h_i)\bold{1}(\Xi)\bold{1}(\tilde{\Xi})-G^i_{ki}(h_i)H_{i1}G^i_{1l}(0)\bold{1}(\Xi)\bold{1}(\tilde{\Xi})\\&-\bold{1}(i\neq 1)G^i_{k1}(h_i)H_{1i}G^i_{il}(0)\bold{1}(\Xi)\bold{1}(\tilde{\Xi})|\,.
\end{align*}
Using that $|H_{1i}|\leq N^{\epsilon-1/2}$ on $\tilde{\Xi}$ and that $\Gamma^i(h_i)\leq N^{\epsilon/3+\delta}$ on $\Xi$
we have that
$$|G^i_{kl}(0)|\leq |G^i_{kl}(h_i)|+ 2N^{\epsilon/3+\delta}N^{\epsilon-1/2} |G^i_{kl}(0)\bold{1}(\Xi)\bold{1}(\tilde{\Xi})|\,.$$
Since $ \epsilon \in (0, \epsilon_0)\,,$ we obtain that $2N^{\epsilon/3+\delta}N^{\epsilon-1/2} \leq 1/2\,,$ that gives that on the events $\Xi$ and $\tilde{\Xi}$ it holds that
$$|G_{kl}^i(0)| \leq 2|G_{kl}^i(h_i)|= 2|G_{kl}| \,, \hspace{2mm} \forall\hspace{2mm} k, l, i \in \llbracket 1, N \rrbracket\,.$$
\end{proof}

We may now introduce the main result of the section.
\begin{lemma}[Concentration of Measure]\label{concentration of measure result}
Let $z \in \bold{S}\,.$ If $|G_{kl}|\prec N^{\delta}\,,$
then \\
$$|G_{kl}-\mathbb{E}_1G_{kl}| \prec \frac{N^{3\delta/2}}{\sqrt{N\eta}}\,, \hspace{4mm} \text{for all } k,l \in \llbracket 1, N \rrbracket\,.$$
\end{lemma}
\begin{proof}
Let us fix arbitrarily $\epsilon \in (0, \epsilon_0)\,.$ We first prove the result for fixed $\eta_0$ as in \eqref{equiv5.1}.

Using the definition of stochastic domination for $|G_{kl}|$ and Lemma \ref{2.1} we have that there exist large $D_1$ and $D_2$ such that $\mathbb{P}(\Xi^c) \leq N^{-D_1}$ and $\mathbb{P}(\tilde{\Xi}^c) \leq N^{-D_2}\,.$\\
Let us define 
$$D\;:=\; \min \{D_1,D_2\}\,.$$

Using that for $G_{kl}=\operatorname{Re}G_{kl}+i\operatorname{Im}G_{kl}\,,$ we have that 
\begin{align}
\mathbb{P}(|G_{kl}-\mathbb{E}_1G_{kl}| \geq \xi ) &\leq
\mathbb{P}(|\operatorname{Re}G_{kl}-\mathbb{E}_1\operatorname{Re}G_{kl}| \geq \frac{\xi}{\sqrt{2}})\nonumber\\&+\mathbb{P}(|\operatorname{Im}G_{kl}-\mathbb{E}_1\operatorname{Im}G_{kl}| \geq \frac{\xi}{\sqrt{2}} )\,.
\end{align} 
Hence, is enough to prove the result only for the real part. The imaginary part is done in the same manner using \eqref{fortheimaginary}\,.\\
Using the estimate in \eqref{equation1} we obtain that 
\begin{align*}
||\operatorname{Re}G_{kl}-\mathbb{E}_1\operatorname{Re}G_{kl}||_{2q}^{2q} \leq 2||V_{\operatorname{Re}}||_{q}^{q}\,.
\end{align*}
From now on, in order to keep the notation simple, we drop the index $\operatorname{Re}$ from $V_{\operatorname{Re}}\,.$ Truncating with respect to the events $\Xi$ and $\tilde{\Xi}$ defined in \eqref{event}\,, we obtain that
\begin{align}\label{termen2}
\mathbb{E}|V|^{q}=\mathbb{E}|V|^{q}\bold{1}(\Xi)\bold{1}(\tilde{\Xi})+\mathbb{E}|V|^{q}\bold{1}(\Xi^c)\bold{1}(\tilde{\Xi})+\mathbb{E}|V|^{q}\bold{1}(\Xi)\bold{1}(\tilde{\Xi}^c)+\mathbb{E}|V|^{q}\bold{1}(\Xi^c)\bold{1}(\tilde{\Xi}^c)\,.
\end{align}
Using the resolvent identity, we obtain that
\begin{equation}
\operatorname{Re}G^i_{kl}(h_i)-\operatorname{Re}G^i_{kl}(0)=-\operatorname{Re}G^i_{k1}(h_i)H_{1i}\operatorname{Re}G^i_{il}(0)-(1-\delta_{1i})\operatorname{Re}G^i_{ki}(h_i)H_{i1}\operatorname{Re}G^i_{1l}(0)\,.
\end{equation}
Taking the absolute values, we obtain the estimate
\begin{multline}\label{grandeestimate}
\sum\limits_{i=1}^N|\operatorname{Re}G^i_{kl}(h_i)-\operatorname{Re}G^i_{kl}(0)|^{2}\leq 
\\
\sum\limits_{i=1}^N2(|\operatorname{Re}G^i_{k1}(h_i)|^2|H_{1i}|^2|\operatorname{Re}G^i_{il}(0)|^2)+\sum\limits_{i=1}^N2(|\operatorname{Re}G^i_{ki}(h_i)|^2|H_{i1}|^2|\operatorname{Re}G^i_{1l}(0)|^2)\,.
\end{multline}
In \eqref{termen2} there are four types of terms.
On $\bold{1}(\Xi)\bold{1}(\tilde{\Xi})$  we use the estimate in \eqref{grandeestimate}. Hence, we have to estimate  
\begin{align}
&\sum\limits_{i=1}^N2(|\operatorname{Re}G^i_{k1}(h_i)|^2|H_{1i}|^2|\operatorname{Re}G^i_{il}(0)|^2)\bold{1}(\Xi)\bold{1}(\tilde{\Xi})\nonumber\\&+\sum\limits_{i=1}^N2(|\operatorname{Re}G^i_{ki}(h_i)|^2|H_{i1}|^2|\operatorname{Re}G^i_{1l}(0)|^2)\bold{1}(\Xi)\bold{1}(\tilde{\Xi})\,.\nonumber\\
\end{align}
Using Lemma \ref{Gamma0} $(i)$ together with $|H_{ij}|^{2}\leq N^{\epsilon-1/2}$ on $\tilde{\Xi}$\,, we obtain that
\begin{align}
&\sum\limits_{i=1}^N2(|\operatorname{Re}G^i_{k1}(h_i)|^2|H_{1i}|^2|\operatorname{Re}G^i_{il}(0)|^2)\bold{1}(\Xi)\bold{1}(\tilde{\Xi})\nonumber\\&+\sum\limits_{i=1}^N2(|\operatorname{Re}G^i_{ki}(h_i)|^2|H_{i1}|^2|\operatorname{Re}G^i_{1l}(0)|^2)\bold{1}(\Xi)\bold{1}(\tilde{\Xi}) \nonumber\\
&\leq \sum\limits_{i=1}^N2(N^{\epsilon/3+\delta})^2(N^{\epsilon-1/2})^2|\operatorname{Re}G^i_{il}(0)|^2\bold{1}(\Xi)\nonumber\\&+\sum\limits_{i=1}^N2|\operatorname{Re}G^i_{ki}(h_i)|^2(N^{\epsilon-1/2})^2(2N^{\epsilon/3+\delta})^2\bold{1}(\Xi)\,.
\end{align}
Using Lemma \ref{2.3} in the first term and using the connection between notations from section $1$ - i.e. $|G_{ki}^i(h_i)|=|G_{ki}|$- for the second term, we conclude that we can use Ward identity \eqref{eq5} to obtain the estimate
\begin{align}
&\sum\limits_{i=1}^N2(N^{\epsilon/3+\delta})^2(N^{\epsilon-1/2})^2|\operatorname{Re}G^i_{il}(0)|^2\bold{1}(\Xi)\nonumber\\&+\sum\limits_{i=1}^N2|\operatorname{Re}G^i_{ki}(h_i)|^2(N^{\epsilon-1/2})^2(2N^{\epsilon/3+\delta})^2\bold{1}(\Xi)\nonumber\\
&\leq 8(N^{\epsilon/3+\delta})^2(N^{\epsilon-1/2})^2\frac{N^{\epsilon/3+\delta}}{\eta_0}+8(N^{\epsilon-1/2})^2(N^{\epsilon/3+\delta})^2\frac{N^{\epsilon/3+\delta}}{\eta_0}\,.
\end{align}
Therefore, we obtain that 
\begin{align}\label{termenul1}
\mathbb{E}|V|^{q}\bold{1}(\Xi)\bold{1}(\tilde{\Xi}) \leq \mathbb{E}\left(8\frac{N^{3\epsilon+3\delta-1}}{\eta_0}+8\frac{N^{3\epsilon+3\delta-1}}{\eta_0}\right)^q=\left(16\frac{N^{3\epsilon+3\delta-1}}{\eta_0}\right)^q\,.
\end{align}
On the remaining three terms, using the brutal estimates $|G_{kl}|\leq N$ and the $L^p$ bounds for $H_{ij}$\,, we obtain finite variations for the random variables formed of monomials or linear combinations of monomials in the entries of $H$ and $G\,.$ Hence, we can apply Cauchy-Schwarz inequality in the form 
$$ \mathbb{E}[X\bold{1}(\Xi^c)] \leq \sqrt{\mathbb{P}[\Xi^c]} \sqrt{\mathbb{E}[X^2]}\,,$$
where $X(H,G)$ is a monomial or a linear combination of monomials in the entries of $H$ and $G\,.$ We do not track the dependence on $H$ and $G$ in our notation.\\
Now we start the analysis of the remaining terms in \eqref{termen2}\,.
For the second term in the truncation, we have that
\begin{align}
\mathbb{E}|V|^{q}\bold{1}(\Xi^c)\bold{1}(\tilde{\Xi}) \leq \sqrt{\mathbb{P}(\Xi^c)}\sqrt{\mathbb{E}|V|^{2q}\bold{1}(\tilde{\Xi})}\,.
\end{align}
Using the brutal estimate $|G_{kl}| \leq N$ on $\Xi^c$ $ \text{for}$ $k, l \in \llbracket 1, N \rrbracket$ and Lemma \ref{Gamma0} $(ii)$ we obtain for the term in the second radical the estimate 
\begin{align}
\mathbb{E}|V|^{2q}\bold{1}(\tilde{\Xi})& = \mathbb{E}|\sum\limits_{i=1}^N(\operatorname{Re}G^i_{kl}(h_i)-\operatorname{Re}G^i_{kl}(0))^2|^{2q}\bold{1}(\tilde{\Xi})\nonumber\\
&\leq |4N^3|^{2q}\,.
\end{align}
Next, using the definition of stochastic domination in the first radical, we obtain the estimate
\begin{equation}\label{termenul2}
\mathbb{E}|V|^q\bold{1}(\Xi^c)\bold{1}(\tilde{\Xi}) \leq 4^{q}N^{3q}N^{-D/2}\,.
\end{equation}
For the third term in the truncation we proceed in the same manner. Hence, we have that
\begin{align}
\mathbb{E}|V|^{q}\bold{1}(\Xi)\bold{1}(\tilde{\Xi}^c) \leq \sqrt{\mathbb{P}(\tilde{\Xi}^c)}\sqrt{\mathbb{E}|V|^{2q}\bold{1}(\Xi)}\,.
\end{align}
Using the advantage of the low probability term in the first radical, we use just the brutal estimate $|G_{kl}| \leq N$ for $k, l \in \llbracket 1, N \rrbracket \,.$ Being the resolvent of a Hermitian matrix, the same estimate applies for $G_{kl}^i(0)\,,$ i.e. $|G_{kl}^i(0)|\leq N\,.$ We finally obtain for the term in the second radical the estimate 
\begin{align}
\mathbb{E}|V|^{2q}\bold{1}(\Xi)& = \mathbb{E}|\sum\limits_{i=1}^N(\operatorname{Re}G^i_{kl}(h_i)-\operatorname{Re}G^i_{kl}(0))^2|^{2q}\bold{1}(\Xi)\nonumber\\
&\leq |4N^3|^{2q}\,.
\end{align}
Now,  using the definition of stochastic domination in the first radical, we obtain the estimate
\begin{equation}\label{termenul3}
\mathbb{E}|V|^q\bold{1}(\Xi)\bold{1}(\tilde{\Xi}^c) \leq 4^{q}N^{3q}N^{-D/2}\,.
\end{equation}
For the fourth term in the truncation, we have that
\begin{align}
\mathbb{E}|V|^{q}\bold{1}(\Xi^c)\bold{1}(\tilde{\Xi}^c) \leq \sqrt{\mathbb{P}(\tilde{\Xi}^c)}\sqrt{\mathbb{E}|V|^{2q}\bold{1}(\Xi^c)}\,.
\end{align}
Using again the brutal estimate $|G_{kl}| \leq N$ on $\Xi^c$ $ \text{for}$ $k, l \in \llbracket 1, N \rrbracket \,,$ and Lemma \ref{Gamma0} $(ii)$ we obtain for the term in the second radical, the estimate 
\begin{align}
\mathbb{E}|V|^{2q}\bold{1}(\Xi^c)& = \mathbb{E}|\sum\limits_{i=1}^N(\operatorname{Re}G^i_{kl}(h_i)-\operatorname{Re}G^i_{kl}(0))^2|^{2q}\bold{1}(\Xi^c)\nonumber\\
&\leq |4N^3|^{2q}\,.
\end{align}
Finally, using the definition of stochastic domination in the first radical, we obtain the estimate
\begin{equation}\label{termenul4}
\mathbb{E}|V|^q\bold{1}(\Xi^c)\bold{1}(\tilde{\Xi}^c) \leq 4^{q}N^{3q}N^{-D/2}\,.
\end{equation}
Plugging in the Chebyshev inequality the terms from \eqref{termenul1}, \eqref{termenul2} \eqref{termenul3} and \eqref{termenul4}  we obtain that
\begin{align}
\mathbb{P}(|\operatorname{Re}G_{kl}-\mathbb{E}_1\operatorname{Re}G_{kl}|\geq \frac{N^{2\epsilon+3\delta/2}}{\sqrt{2N\eta_0}})&\leq \frac{\mathbb{E}|\operatorname{Re}G_{kl}-\mathbb{E}_1\operatorname{Re}G_{kl}|^{2q}}{\left(\frac{N^{2\epsilon+3\delta/2}}{\sqrt{2N\eta_0}}\right)^{2q}} \nonumber\\
&\leq \left(\frac{\frac{16^qN^{3q\epsilon+3q\delta}}{(N\eta_0)^q}}{{\frac{N^{4\epsilon q+3q\delta}}{(2N\eta_0)^q}}}\right)+ \left(3\frac{4^{q}N^{3q}N^{-D/2}}{\frac{N^{4\epsilon q+3q\delta}}{(2N\eta_0)^q}}   \right)\,.
\end{align}
Hence, for $\epsilon \in (0, \epsilon_0)$  and (large) $D>0\,,$ we obtain the bound
\begin{align*}
\mathbb{P}(|\operatorname{Re}G_{kl}-\mathbb{E}_1\operatorname{Re}G_{kl}| \geq \frac{N^{2\epsilon+3\delta/2}}{\sqrt{2N\eta_0}} ) \leq 32^qN^{-q\epsilon}+8^{q}3N^{3q-3q\delta-4q\epsilon+q-\frac{D}{2}}\eta_0^q\,.
\end{align*}
In the same manner, we obtain that 
\begin{align*}
\mathbb{P}(|\operatorname{Im}G_{kl}-\mathbb{E}_1\operatorname{Im}G_{kl}|\geq \frac{N^{2\epsilon+3\delta/2}}{\sqrt{2N\eta_0}}) \leq 32^qN^{-q\epsilon}+8^{q}3N^{3q-3q\delta-4q\epsilon+q-\frac{D}{2}}\eta_0^q\,.
\end{align*}
Combining the two estimates, we obtain that
\begin{align}\label{bla}
\mathbb{P}(|G_{kl}-\mathbb{E}_1G_{kl}|\geq \frac{N^{2\epsilon+3\delta/2}}{\sqrt{N\eta_0}})&\leq\nonumber\\
\mathbb{P}(|\operatorname{Re}G_{kl}-\mathbb{E}_1\operatorname{Re}G_{kl}| \geq \frac{N^{2\epsilon+3\delta/2}}{\sqrt{2N\eta_0}})&+ \mathbb{P}(|\operatorname{Im}G_{kl}-\mathbb{E}_1\operatorname{Im}G_{kl}| \geq \frac{N^{2\epsilon+3\delta/2}}{\sqrt{2N\eta_0}})\nonumber\\
& \leq 32^q2N^{-q\epsilon}+8^{q}6N^{3q-3q\delta-4q\epsilon+q-\frac{D}{2}}\eta_0^q\,.
\end{align}
First, for fixed $ \delta \in (0, \delta_0) $ given $\epsilon \in (0, \epsilon_0)$ and large $D>0,$ there is a  large enough $q$, such that $32^q2N^{-q\epsilon}$ is bounded by $ 8^{q}6N^{3q-3q\delta-4q\epsilon+q-\frac{D}{2}}\eta_0^q\,, $ i.e. we choose $q$ such that
\begin{align*}
8^{q}6N^{3q/2-3q\delta-4q\epsilon+q-\frac{D}{2}}\eta_0^q \geq 32^q2N^{-q\epsilon}\,.
\end{align*}
Using the monotonicity of the logarithm, we obtain that
\begin{align}
&q\log 8+ \log 6+\left(\frac{3q}{2}-3q\delta-4q\epsilon+q-\frac{D}{2}\right)\log N +q \log \eta_0 \nonumber\\&\geq q \log 32 -q\epsilon \log N+\log 2\,.
\end{align}
i.e.
\begin{align}
q \geq \frac{\frac{D}{2}\log N-\log 3}{\log\frac{1}{4}+\frac{3\log N}{2}-3\delta\log N-4\epsilon \log N+\log N +\log \eta_0+\epsilon \log N}\,.
\end{align}
Second, since $\epsilon \in (0, \epsilon_0)$ and $D>0$ were arbitrary, the proof is complete for fixed $\eta_0 \,.$\\ 
It remains to prove that the concentration bound holds for all $z$ as in \eqref{equiv5.1}\,.\\
We follow the strategy of Lemma $5.5$ in \cite{bauerschmidt2017local}\,. For this, set
$$ \eta_l=\eta_0+l/N^{4} \,, \hspace{4mm} l \in \llbracket 1, N^5 \rrbracket\,,$$
and $z_l=E+i\eta_l$\,. 
Since the bound in the concentration result holds uniformly for any $\eta \geq \eta_0\,,$ we have by an union bound, that \eqref{bla} holds simultaneously at all $z$ with $l \in \llbracket 1, N^5 \rrbracket$\,.\\
Since $(\eta_l)_l$ is a $ 1/N^{4}$-net of $\llbracket \eta_0, \eta_0+N \rrbracket$ and $G_{ij}$ is Lipschitz continuous  with the constant $1/\eta^2 \leq N^2 $\,, the claim follows.\\
Repeating the proof for $k \,, l \in \llbracket 1, N \rrbracket\,,$ we obtain the result for every entry of the resolvent.
\end{proof}

\section{Analysis of the Average}

In this section we use the Concentration of Measure Lemma \ref{concentration of measure result} and the resolvent expansion to give a bound for $1+sz+s^2 \,,$ where $s\;:=\; \frac{1}{N}\sum_{i=1}^NG_{ii}\,.$ First, we prove a Lemma that gives bounds for products of the entries of the resolvent. After that, we use the resolvent expansion along with truncations of the summations that appear, to obtain bounds for $1+sz+s^2$ at a fixed scale. Furthermore, as in the previous section, we extend our analysis on smaller scales.

Throughout this section,  we fix the parameter $\delta \in (0, \delta_0)\,,$ where $\delta_0 \leq \frac{\gamma}{3}\,.$
For the fixed value of $ \delta \in (0, \delta_0)\,, $  let $\epsilon_0$ be such that $ \epsilon_0 \leq \frac{3}{4}\left( \frac{1}{2}-\delta-\frac{\log 4}{\log N} \right)\,.$

Using the definition of the resolvent 
\begin{align*}
(H-z)G=\bold{1}\,,
\end{align*}
we obtain 
\begin{equation}\label{eq2.1}
1+zG_{11}\;=\;\sum\limits_{i=1}^NH_{1i}G_{i1}\,.
\end{equation}

In order to prove the main result of the section we need the following lemma.
\begin{lemma}\label{lemma3.1}
 Let $\eta>0$ be a fixed real number. Let $X_i , i \in \llbracket 1, 2\rrbracket\,,$ be two complex random variables, such that $X_i-\mathbb{E}_1X_{i}=O_{\prec}\left( \frac{N^{3/2\delta}}{\sqrt{N\eta}}\right)$ and $ |X_i| \prec N^{\delta}\,,$ $\forall$ $ i \in \llbracket 1, 2 \rrbracket$\,.\\
Then, we have
\begin{equation}
\mathbb{E}_1X_1X_2 -X_1X_2 =O_{\prec}\left( \frac{N^{5\delta/2}}{\sqrt{N\eta}} \right)\,.
\end{equation}
\end{lemma}
\begin{proof}
We start the proof by taking the expectation $\mathbb{E}_1$ in the  identity 
\begin{align}
X_1X_2=X_1(X_2-\mathbb{E}_1X_2)+X_1\mathbb{E}_1X_2\,.
\end{align}
We further have that
\begin{align}
\mathbb{E}_1X_1X_2-\mathbb{E}_1X_1\mathbb{E}_1X_2=\mathbb{E}_1X_1(X_2-\mathbb{E}_1X_2)\,.
\end{align}
Using the assumptions and  the result of Lemma $3.4 (ii)$ from \cite{benaych2016lectures} for the pairs $X_2-\mathbb{E}_1X_2=O_{\prec}\left(\frac{N^{3\delta/2}}{\sqrt{N\eta}}\right)$ and $|X_1| \prec N^{\delta}\,,$ we obtain that
\begin{align}
\mathbb{E}_1X_1X_2-\mathbb{E}_1X_1\mathbb{E}_1X_2 =O_{\prec}\left( \frac{N^{\delta}N^{3\delta/2}}{\sqrt{N\eta}}\right)\,.
\end{align}
Repeating the procedure for $X_1$ we obtain that
\begin{align}
\mathbb{E}_1X_1X_2+(X_1-\mathbb{E}_1X_1)\mathbb{E}_1X_2-X_1\mathbb{E}_1X_2 =O_{\prec}\left( \frac{N^{\delta}N^{3\delta/2}}{\sqrt{N\eta}}\right)\,.
\end{align}
Now, using the assumptions and  the result of Lemma $3.4$ $(ii)$ from \cite{benaych2016lectures} for $X_1-\mathbb{E}_1X_1=O_{\prec}\left(\frac{N^{3\delta/2}}{\sqrt{N\eta}}\right)$ and $|X_2| \prec N^{\delta}\,,$ we obtain that
\begin{align}\label{bleah}
\mathbb{E}_1X_1X_2-X_1\mathbb{E}_1X_2=O_{\prec}\left(\frac{N^{5\delta/2}}{\sqrt{N\eta}}\right)\,.
\end{align}
Finally, by rewriting \eqref{bleah} as
\begin{align}
\mathbb{E}_1X_1X_2+X_1(X_2-\mathbb{E}_1X_2)-X_1X_2=O_{\prec}\left(\frac{N^{5\delta/2}}{\sqrt{N\eta}}\right)\,.
\end{align}
and repeating the procedure, we conclude the proof.
\end{proof}

The following proposition is the main result of this section.
\begin{proposition}\label{prop2.2}
Let $ \epsilon \in (0, \epsilon_0)$ and let $s:=\frac{1}{N}\sum\limits_{j=1}^NG_{jj}$\,. If $|G_{kl}| \prec N^{\delta}\,,$ for $k ,l \in \llbracket 1, N \rrbracket\,,$ it follows that, for all $ z \in \bold{S}\,,$ we have
\begin{equation}
1+zs+s^2=O_{\prec}\left(\frac{(1+|z|)N^{5\delta}}{\sqrt{N\eta}}\right)\,.
\end{equation}
\end{proposition}
\begin{proof}
Let us fix $\epsilon \in (0, \epsilon_0)$\,. We first prove the result for fixed $\eta_0$ as in \eqref{equiv5.1}.

Using the definition of stochastic domination for $|G_{kl}|$ and Lemma \ref{2.1} we have that there exist large $D_1$ and $D_2$ such that $\mathbb{P}(\Xi^c) \leq N^{-D_1}$ and $\mathbb{P}(\tilde{\Xi}^c) \leq N^{-D_2}\,.$ Let us define 
$$D\;:=\; \min \{D_1,D_2\}\,.$$
Using the resolvent expansion in \eqref{eq2.1}\,, we obtain that
\begin{equation}
1+zG_{11}=\sum\limits_{i=1}^NH_{1i}G_{i1}=\sum\limits_{i=1}^NH_{1i}\left(G_{i1}^{(1i)}-G_{ii}^{(1i)}H_{i1}G_{11}-G_{i1}^{(1i)}H_{1i}G_{i1}\bold{1}(i\neq 1)\right)\,.
\end{equation}
Performing another resolvent expansion for $G_{11}$\,, we obtain that 
\begin{align}\label{3.8}
&\sum\limits_{i=1}^NH_{1i}\left(G_{i1}^{(1i)}-G_{ii}^{(1i)}H_{i1}G_{11}-G_{i1}^{(1i)}H_{1i}G_{i1}\bold{1}(i\neq 1)\right)=\nonumber\\
&\sum\limits_{i=1}^NH_{1i}G_{i1}^{(1i)}-\sum\limits_{i=1}^N|H_{i1}|^{2}G_{ii}^{(1i)}\left(G_{11}^{(1i)}-G_{1i}^{(1i)}H_{i1}G_{11}-G_{11}^{(1i)}H_{1i}G_{i1}\bold{1}(i\neq 1)\right)\nonumber\\&+O\left(\sum\limits_{i=1}^N|G_{i1}^{(1i)}||G_{i1}||H_{1i}|^2\right)\,.
\end{align}
Furthermore, we rewrite the right hand side term in \eqref{3.8} in the form
\begin{align}
&\sum\limits_{i=1}^NH_{i1}G_{i1}^{(1i)}-\sum\limits_{i=1}^N|H_{i1}|^2G_{ii}^{(1i)}G_{11}^{(1i)}\\
&+O\left( \sum\limits_{i=1}^N|H_{i1}|^2(|G_{1i}^{(1i)}||G_{ii}^{(1i)}||G_{11}||H_{i1}|+|G_{11}^{(1i)}||G_{i1}||G_{ii}^{(1i)}||H_{1i}|\right)\nonumber\\&+O\left(\sum\limits_{i=1}^N|H_{1i}|^2|G_{i1}^{(1i)}||G_{i1}|\right)\,.
\end{align}
Using that $H_{1i}$ and $G_{i1}^{(1i)}$ are independent and taking the expectation $\mathbb{E}_1\,,$ we further obtain that
\begin{align}\label{sumstobe}
&1+z\mathbb{E}_{1}G_{11}=-\mathbb{E}_1\sum\limits_{i=1}^N|H_{i1}|^2G_{ii}^{(1i)}G_{11}^{(1i)}\nonumber\\
&+O\left(\mathbb{E}_1\sum\limits_{i=1}^N|H_{i1}|^2(|G_{1i}^{(1i)}||G_{ii}^{(1i)}||G_{11}||H_{i1}|+|G_{11}^{(1i)}||G_{i1}||G_{ii}^{(1i)}||H_{1i}|\right)\nonumber\\&+ O\left(\mathbb{E}_{1}\sum\limits_{i=1}^N|H_{i1}|^2|G_{i1}^{(1i)}||G_{i1}|\right)\,.
\end{align}
Using the connection between notations from section $1\,,$ we have that
$|G_{ii}^{(1i)}|=|G^i_{ii}(0)|,$
$|G_{11}^{(1i)}|=|G^i_{11}(0)|$
and $|G_{1i}^{(1i)}|=|G^i_{1i}(0)|\,.$
Truncating the  summations in \eqref{sumstobe} with respect to the events \eqref{event}, we obtain that
\begin{align*}
(&1+z\mathbb{E}_{1}G_{11})(\bold{1}(\Xi)+ \bold{1}(\Xi^c))(\bold{1}(\tilde{\Xi})+\bold{1}(\tilde{\Xi}^c))\\&=-\mathbb{E}_1\sum\limits_{i=1}^N|H_{i1}|^2G_{ii}^{(1i)}G_{11}^{(1i)}(\bold{1}(\Xi)+ \bold{1}(\Xi^c))(\bold{1}(\tilde{\Xi})+\bold{1}(\tilde{\Xi}^c))\\
&+O(\mathbb{E}_1\sum\limits_{i=1}^N|H_{1i}|^2(|G_{1i}^{(1i)}||G_{ii}^{(1i)}||G_{11}||H_{i1}|\\&+|G_{11}^{(1i)}||G_{i1}||G_{ii}^{(1i)}||H_{1i}|)(\bold{1}(\Xi)+ \bold{1}(\Xi^c))(\bold{1}(\tilde{\Xi})+\bold{1}(\tilde{\Xi}^c))\\
&+\mathbb{E}_{1}\sum\limits_{i=1}^N|H_{1i}|^2|G_{i1}^{(1i)}||G_{i1}|)(\bold{1}(\Xi)+ \bold{1}(\Xi^c))(\bold{1}(\tilde{\Xi})+\bold{1}(\tilde{\Xi}^c)))\,.
\end{align*}
We investigate the terms that contain absolute values of the resolvent entries. We divide them into two categories depending on the numbers of entries of $H$ that they contain. The terms in the same category will be estimated in the same manner. After doing all the possible truncations there are four truncated terms that appear for each term of each category. This leads to eight truncated terms for the first category and four truncated terms for the second category.\\
An example of the four truncated terms that appear for a term in the first category is
\begin{align}
&\mathbb{E}_1\sum\limits_{i=1}^N|H_{1i}|^2(|G_{1i}^{(1i)}||G_{ii}^{(1i)}||G_{11}||H_{i1}|)(\bold{1}(\Xi)+ \bold{1}(\Xi^c))(\bold{1}(\tilde{\Xi})+\bold{1}(\tilde{\Xi}))=\nonumber\\
&\mathbb{E}_1\sum\limits_{i=1}^N|H_{1i}|^2(|G_{1i}^{(1i)}||G_{ii}^{(1i)}||G_{11}||H_{i1}|)\bold{1}(\Xi)\bold{1}(\tilde{\Xi})\nonumber\\&+\mathbb{E}_1\sum\limits_{i=1}^N|H_{1i}|^2(|G_{1i}^{(1i)}||G_{ii}^{(1i)}||G_{11}||H_{i1}|)\bold{1}(\Xi^c)\bold{1}(\tilde{\Xi})\nonumber\\
&+\mathbb{E}_1\sum\limits_{i=1}^N|H_{1i}|^2(|G_{1i}^{(1i)}||G_{ii}^{(1i)}||G_{11}||H_{i1}|)\bold{1}(\Xi)\bold{1}(\tilde{\Xi}^c)\nonumber\\&+\mathbb{E}_1\sum\limits_{i=1}^N|H_{1i}|^2(|G_{1i}^{(1i)}||G_{ii}^{(1i)}||G_{11}||H_{i1}|)\bold{1}(\Xi^c)\bold{1}(\tilde{\Xi}^c)\,.
\end{align}
We first estimate $\mathbb{E}_1\sum\limits_{i=1}^N|H_{1i}|^2(|G_{1i}^{(1i)}||G_{ii}^{(1i)}||G_{11}||H_{i1}|)\bold{1}(\Xi)\bold{1}(\tilde{\Xi})\,.$
Using the definition of the events $\Xi$ and $\tilde{\Xi}\,,$ 
we prove the finite variations for the random variables given by linear combinations of monomials of $H$ and $G\,.$ Hence, we can apply again Cauchy-Schwarz inequality in the form 
$$ \mathbb{E}[X\bold{1}(\Xi)] \leq \sqrt{\mathbb{P}[\Xi]} \sqrt{\mathbb{E}[X^2]}\,,$$
where $X(H,G)$ is a linear combination of monomials in the entries of $H$ and $G\,.$\\
Let us consider the linear combination 
\begin{align}
X_0\;:=\; \sum\limits_{i=1}^N|H_{1i}|^2(|G_{1i}^{(1i)}||G_{ii}^{(1i)}||G_{11}||H_{i1}|)\bold{1}(\tilde{\Xi})\,.
\end{align}
We have that
\begin{align}
&\mathbb{E}_1\sum\limits_{i=1}^N|H_{1i}|^2(|G_{1i}^{(1i)}||G_{ii}^{(1i)}||G_{11}||H_{i1}|)\bold{1}(\Xi)\bold{1}(\tilde{\Xi})\nonumber\\&\leq \sqrt{\mathbb{P}(\Xi)}\sqrt{\mathbb{E}_1(\sum\limits_{i=1}^N |H_{1i}|^2|G_{1i}^{(1i)}||G_{ii}^{(1i)}||G_{11}||H_{i1}|\bold{1}(\tilde{\Xi}))^2}\,.
\end{align}
Using that $\mathbb{P}$ is a probability measure, that $\eta_0=N$ , that $(\sum\limits_{i=1}^N|a_i|)^2 \leq N\sum\limits_{i=1}^N|a_i|^2\,,$ for $a_i \in \mathbb{C}\,,$ Lemma \ref{Gamma0}, Lemma \ref{2.3} and Ward identity we further obtain that
\begin{align}
&\sqrt{\mathbb{P}(\Xi)}\sqrt{\mathbb{E}_1(\sum\limits_{i=1}^N |H_{1i}|^2|G_{1i}^{(1i)}||G_{ii}^{(1i)}||G_{11}||H_{i1}|\bold{1}(\tilde{\Xi}))^2}\nonumber\\
&\leq\sqrt{\mathbb{E}_1 N(N^{\epsilon-1/2})^6(2N^{\epsilon/3+\delta})^2(N^{\epsilon/3+\delta})^2\sum\limits_{i=1}^N|G^{(1i)}_{1i}|^2}\nonumber\\
&\leq 4N^{23\epsilon/6+5\delta/2-3/2}\,.
\end{align}
On the remaining three terms, we use again the brutal estimates $|G_{kl}| \leq N$ and the $L^p$ bounds on $H\,.$ In this manner we prove the finite variations for the random variables given by monomials of $H$ and $G\,.$ Hence, we can apply Cauchy-Schwarz inequality in the form 
$$ \mathbb{E}[X\bold{1}(\Xi^c)] \leq \sqrt{\mathbb{P}[\Xi^c]} \sqrt{\mathbb{E}[X^2]}\,,$$
where $X(H,G)$ is a monomial in the entries of $H$ and $G\,.$ 
All the remaining three type of terms are estimated in the same manner using this strategy. For all of them, we use the connection between notations in section $1$ to get that $|G_{kl}^{(1i)}|=|G^i_{kl}(0)|\,.$ Via this identification, we use the bounds in Lemma \ref{Gamma0} for $|G_{kl}^{(1i)}|\,.$ We work the details for only two of them, namely $ \mathbb{E}_1\sum\limits_{i=1}^N|H_{1i}|^2(|G_{1i}^{(1i)}||G_{ii}^{(1i)}||G_{11}||H_{i1}|)\bold{1}(\Xi^c)\bold{1}(\tilde{\Xi})$ and $\mathbb{E}_1\sum\limits_{i=1}^N|H_{1i}|^2(|G_{1i}^{(1i)}||G_{ii}^{(1i)}||G_{11}||H_{i1}|)\bold{1}(\Xi)\bold{1}(\tilde{\Xi}^c)\,.$ The last term is estimated in the same manner as  $\mathbb{E}_1\sum\limits_{i=1}^N|H_{1i}|^2(|G_{1i}^{(1i)}||G_{ii}^{(1i)}||G_{11}||H_{i1}|)\bold{1}(\Xi^c)\bold{1}(\tilde{\Xi})$ by using the brutal estimates $|G_{kl}| \leq N$ and $|G_{kl}^{(1i)}|\leq N\,.$
In order to estimate the following quantity
$ \mathbb{E}_1\sum\limits_{i=1}^N|H_{1i}|^2(|G_{1i}^{(1i)}||G_{ii}^{(1i)}||G_{11}||H_{i1}|)\bold{1}(\Xi^c)\bold{1}(\tilde{\Xi})$, let us define the monomial
$$X_1(H,G)\;:=\; |H_{1i}|^2(|G_{1i}^{(1i)}||G_{ii}^{(1i)}||G_{11}||H_{i1}|)\bold{1}(\tilde{\Xi})\,.$$ 
Using Cauchy-Schwarz inequality for the monomial $X_1\,,$ the brutal estimate $|G_{kl}|\leq N$ on $\Xi^c$ and Lemma \ref{Gamma0}\,, we obtain that
\begin{align}
&\sum\limits_{i=1}^N\mathbb{E}_1|H_{1i}|^2(|G_{1i}^{(1i)}||G_{ii}^{(1i)}||G_{11}||H_{i1}|)\bold{1}(\Xi^c)\bold{1}(\tilde{\Xi})\nonumber\\ &\leq 
\sum\limits_{i=1}^N\sqrt{\mathbb{P}(\Xi^c)}\sqrt{\mathbb{E}_1|H_{1i}|^4(|G_{1i}^{(1i)}|^2|G_{ii}^{(1i)}|^2|G_{11}|^2|H_{i1}|^2)\bold{1}(\tilde{\Xi})}\nonumber\\
&\leq \sum\limits_{i=1}^NN^{-D/2}\sqrt{N^{6}\mathbb{E}_1|H_{1i}|^6}\nonumber\\
&\leq N^{-D/2}C_6^3N^{5/2}\,.
\end{align}
where in the last inequality we used the condition $(iii)$ in the definition of the Wigner matrices for the $L^6$ norm\,.\\
For estimating $\mathbb{E}_1\sum\limits_{i=1}^N|H_{1i}|^2(|G_{1i}^{(1i)}||G_{ii}^{(1i)}||G_{11}||H_{i1}|)\bold{1}(\Xi)\bold{1}(\tilde{\Xi}^c)$  we define the monomial 
$$ X_2(H,G)\;:=\; |H_{1i}|^2(|G_{1i}^{(1i)}||G_{ii}^{(1i)}||G_{11}||H_{i1}|)\bold{1}(\Xi)\,,$$
and we take advantage of the low probability term as in section 2, i.e. we brutally estimate $|G_{kl}| \leq N$ and $|G_{kl}^{(1i)}|\leq N\,.$ Using the Cauchy-Schwarz inequality, we obtain that
\begin{align}
\mathbb{E}_1\sum\limits_{i=1}^N|H_{1i}|^2(|G_{1i}^{(1i)}||G_{ii}^{(1i)}||G_{11}||H_{i1}|)\bold{1}(\Xi)\bold{1}(\tilde{\Xi}^c)\leq N^{-D/2}C_6^3N^{5/2}\,. 
\end{align}
For the four corresponding truncated terms we obtain the final  estimate $4N^{23\epsilon/6+5\delta/2-3/2}+ 3N^{-D/2}C_6^3N^{5/2}\,.$\\
There are four truncated terms for each term of the first category so the final estimate for the terms in the first category is  
\begin{align}
8N^{23\epsilon/6+5\delta/2-3/2}+ 6N^{-D/2}C_6^3N^{5/2}\,.
\end{align}
We proceed exactly the same for the term in the second category.
The truncated terms that appear for the term in the second category are
\begin{align}
&\mathbb{E}_{1}\sum\limits_{i=1}^N(|H_{1i}|^2|G_{i1}^{(1i)}||G_{i1}|)(\bold{1}(\Xi)+ \bold{1}(\Xi^c))(\bold{1}(\tilde{\Xi})+\bold{1}(\tilde{\Xi}^c))=\nonumber\\
&\mathbb{E}_{1}\sum\limits_{i=1}^N(|H_{1i}|^2|G_{i1}^{(1i)}||G_{i1}|)\bold{1}(\Xi)\bold{1}(\tilde{\Xi}) +\mathbb{E}_{1}\sum\limits_{i=1}^N(|H_{1i}|^2|G_{i1}^{(1i)}||G_{i1}|)\bold{1}(\Xi^c)\bold{1}(\tilde{\Xi})+\nonumber\\
&\mathbb{E}_{1}\sum\limits_{i=1}^N(|H_{1i}|^2|G_{i1}^{(1i)}||G_{i1}|)\bold{1}(\Xi)\bold{1}(\tilde{\Xi}^c) +\mathbb{E}_{1}\sum\limits_{i=1}^N(|H_{1i}|^2|G_{i1}^{(1i)}||G_{i1}|)\bold{1}(\Xi^c)\bold{1}(\tilde{\Xi}^c)\,.
\end{align}
Via Cauchy-Schwarz inequality applied for 
$$X \;:=\; \sum\limits_{i=1}^N|H_{1i}|^2||G_{i1}^{(1i)}||G_{i1}|\bold{1}(\tilde{\Xi})\,,$$
we obtain the estimate
\begin{align}
\mathbb{E}_{1}\sum\limits_{i=1}^N(|H_{1i}|^2|G_{i1}^{(1i)}||G_{i1}|)\bold{1}(\Xi)\bold{1}(\tilde{\Xi}) \leq \sqrt{\mathbb{P}(\Xi)}\sqrt{\mathbb{E}_1X^2}\,.
\end{align}
Using that $\mathbb{P}$ is a probability measure we bound with $1$ the first radical\,. For the second radical, using the estimate $|\sum\limits_{i=1}^Na_i|^2 \leq N\sum\limits_{i=1}^N|a_i|^2\,,$ for $a_i \in \mathbb{C}\,,$ along with Lemma \ref{Gamma0}, Lemma \ref{2.3} and Ward identity for $\eta_0=N$, we get the estimate
\begin{align}
\sqrt{\mathbb{E}_1 N\sum\limits_{i=1}^N|H_{1i}|^4||G_{i1}^{(1i)}|^2|G_{i1}|^2\bold{1}(\tilde{\Xi}) } &\leq N^{\epsilon/3+\delta}N^{2\epsilon-1}\sqrt{4N\mathbb{E}_1\sum\limits_{i=1}^N|G_{i1}|^2}\nonumber\\
&\leq 2N^{15\epsilon/6+3\delta/2-1}\,.
\end{align} 
Proceeding in the same manner as for the terms in the first category (i.e. using the Cauchy-Schwarz inequality only for monomials) we obtain for the remaining three terms the estimate $3C_4^2N^{2}N^{-D/2}\,,$
where $C_4$ is the constant corresponding to the $L^4$ norm estimate as in the definition of Wigner matrices.
For our particular example of term in the second category, we obtain the final estimate
\begin{align}
 2N^{15\epsilon/6+3\delta/2-1}+3C_4^2N^2N^{-D/2}\,.
\end{align}
Combining the two final estimates for the terms in the first and second category, we obtain that
\begin{align}\label{term3}
&1+z\mathbb{E}_{1}G_{11} =-\mathbb{E}_1\sum\limits_{i=1}^N|H_{1i}|^2G_{ii}^{(1i)}G_{11}^{(1i)}\nonumber\\
&+O\left(8N^{23\epsilon/6+5\delta/2-3/2} + 6N^{-D/2}C_6^3N^{5/2}+2N^{15\epsilon/6+3\delta/2-1}+3C_4^2N^{2}N^{-D/2}\right)\,.
\end{align}
Using the linearity of the conditional expectation and that $H_{1i}$ and  $H_{i1}$ are independent of $G_{ii}^{(1i)}$ and of $G_{11}^{(1i)}$ we obtain that 
\begin{align}
-\mathbb{E}_1\sum\limits_{i=1}^NH_{1i}H_{i1}G_{ii}^{(1i)}G_{11}^{(1i)}=-\frac{1}{N}\sum\limits_{i=1}^N\mathbb{E}_1G_{ii}^{(1i)}G_{11}^{(1i)}\,.
\end{align} 
Next, we develop a procedure for estimating the product $G^{(1i)}_{ii}G^{(1i)}_{11}$  in terms of $G_{ii}G_{11}$ and other bounds.
For this, we redo the resolvent identity in \eqref{term3} for 
\begin{align}
\tilde{H}=H+\bar{\Delta}\,,
\end{align}
with $\tilde{H}=H^{(1i)}$ and $\bar{\Delta}^{(1i)} = - \Delta^{(1i)}\,.$\\
Using the resolvent identity, we obtain that
\begin{align}
G_{ii}^{(1i)}=G_{ii}+G_{ii}H_{i1}G_{1i}^{(1i)}+\bold{1}(i\neq 1) G_{i1}H_{1i}G_{ii}^{(1i)}\,,
\end{align}
and
\begin{align}
G_{11}^{(1i)}=G_{11}+G_{1i}H_{i1}G_{11}^{(1i)}+\bold{1}(i\neq 1)G_{11}H_{1i}G_{i1}^{(1i)}\,.
\end{align}
Let $$\zeta_i :=(G_{ii}+G_{ii}H_{i1}G_{1i}^{(1i)}+\bold{1}(i\neq 1) G_{i1}H_{1i}G_{ii}^{(1i)})(G_{11}+G_{1i}H_{i1}G_{11}^{(1i)}+\bold{1}(i\neq 1)G_{11}H_{1i}G_{i1}^{(1i)}).$$
It follows that
\begin{align}
-\frac{1}{N}\sum\limits_{i=1}^N\mathbb{E}_1G_{ii}^{(1i)}G_{11}^{(1i)}=-\frac{1}{N}\sum\limits_{i=1}^N\mathbb{E}_1\zeta_i\,.
\end{align}
We further obtain that 
\begin{align}\label{longequation}
&-\frac{1}{N}\sum\limits_{i=1}^N\mathbb{E}_1G_{ii}^{(1i)}G_{11}^{(1i)}=-\frac{1}{N}\sum\limits_{i=1}^N\mathbb{E}_1G_{ii}G_{11} \nonumber\\
&+O( \frac{1}{N}\sum\limits_{i=1}^N\mathbb{E}_1|G_{ii}||G_{1i}||H_{i1}||G_{11}^{(1i)}| +\frac{1}{N}\sum\limits_{i=1}^N\mathbb{E}_1|G_{11}||G_{ii}||H_{i1}||G_{1i}^{(1i)}|\nonumber\\
&+\frac{1}{N}\sum\limits_{i=1}^N\mathbb{E}_1|G_{ii}||G_{11}||H_{1i}||G_{i1}^{(1i)}|+\frac{1}{N}\sum\limits_{i=1}^N\mathbb{E}_1|G_{11}||G_{i1}||H_{1i}||G_{ii}^{(1i)}|\nonumber\\
&+\frac{1}{N}\sum\limits_{i=1}^N\mathbb{E}_1|G_{ii}||H_{i1}||G_{1i}^{(1i)}||G_{1i}||H_{i1}||G_{11}^{(1i)}|\nonumber\\&+\frac{1}{N}\sum\limits_{i=1}^N\mathbb{E}_1|G_{ii}||H_{i1}||G_{1i}^{(1i)}||G_{11}||H_{1i}||G_{i1}^{(1i)}|\nonumber\\
&+\frac{1}{N}\sum\limits_{i=1}^N\mathbb{E}_1|G_{i1}||H_{1i}|G_{ii}^{(1i)}||G_{1i}||H_{i1}||G_{1i}^{(1i)}|\nonumber\\&+\frac{1}{N}\sum\limits_{i=1}^N\mathbb{E}_1G_{i1}||H_{1i}||G_{ii}^{(1i)}||G_{11}||H_{1i}||G_{i1}^{(1i)}|) \,.
\end{align}
The terms  in \eqref{longequation} can be divided into two categories: terms in which entries of $H$ appear only once, and terms in which entries of $H$ appear twice. The terms in the same category will be estimated in the same manner. After doing all the possible truncations there are four truncated terms that appear for each term of each category. This leads to sixteen truncated terms for the first category and sixteen truncated terms for the second category.\\
An example of the truncated terms that appear for a term in the first category is
\begin{align}
&\frac{1}{N}\sum\limits_{i=1}^N\mathbb{E}_1|G_{ii}||G_{1i}||H_{i1}||G_{11}^{(1i)}|(\bold{1}(\Xi)+ \bold{1}(\Xi^c))(\bold{1}(\tilde{\Xi})+\bold{1}(\tilde{\Xi}^c))\nonumber\\
&=\frac{1}{N}\sum\limits_{i=1}^N\mathbb{E}_1|G_{ii}||G_{1i}||H_{i1}||G_{11}^{(1i)}|\bold{1}(\Xi)\bold{1}(\tilde{\Xi})\nonumber\\&+\frac{1}{N}\sum\limits_{i=1}^N\mathbb{E}_1|G_{ii}||G_{1i}||H_{i1}||G_{11}^{(1i)}|\bold{1}(\Xi^c)\bold{1}(\tilde{\Xi})\nonumber\\
&+\frac{1}{N}\sum\limits_{i=1}^N\mathbb{E}_1|G_{ii}||G_{1i}||H_{i1}||G_{11}^{(1i)}|\bold{1}(\Xi)\bold{1}(\tilde{\Xi}^c)\nonumber\\&+\frac{1}{N}\sum\limits_{i=1}^N\mathbb{E}_1|G_{ii}||G_{1i}||H_{i1}||G_{11}^{(1i)}|\bold{1}(\Xi^c)\bold{1}(\tilde{\Xi}^c)\,.
\end{align}
Like before, using Cauchy-Schwarz inequality along with Lemma \ref{Gamma0}\,, Lemma \ref{2.3} and Ward identity we obtain for $\frac{1}{N}\sum\limits_{i=1}^N\mathbb{E}_1|G_{ii}||G_{1i}||H_{i1}||G_{11}^{(1i)}|\bold{1}(\Xi)\bold{1}(\tilde{\Xi})$ the estimate
\begin{align}
&\frac{1}{N}\sum\limits_{i=1}^N\mathbb{E}_1|G_{ii}||G_{1i}||H_{i1}||G_{11}^{(1i)}|\bold{1}(\Xi)\bold{1}(\tilde{\Xi})\nonumber\\&\leq \frac{1}{N}\sqrt{\mathbb{P}(\Xi)}\sqrt{ \mathbb{E}_1(\sum\limits_{i=1}^N |G_{ii}||G_{1i}||H_{i1}||G_{11}^{(1i)}|\bold{1}(\tilde{\Xi}))^2}\nonumber\\
&\leq \frac{2}{N}N^{2\epsilon/3+2\delta}N^{\epsilon-1/2}\sqrt{N\sum\limits_{i=1}^N|G_{1i}|^2}\nonumber\\
&=2N^{11\epsilon/3+5\delta/2-3/2}\,.
\end{align}
On the remaining three terms we apply Cauchy-Schwarz inequality like before.
All the remaining three type of terms are estimated in the same manner. For all of them, we use the connection between notations from section $1$ to get that $|G_{kl}^{(1i)}|=|G^i_{kl}(0)|\,.$ Via this identification we use the bounds in Lemma \ref{Gamma0} for $|G_{kl}^{(1i)}|\,.$ We work the details for only one of them, namely $\frac{1}{N}\sum\limits_{i=1}^N\mathbb{E}_1|G_{ii}||G_{1i}||H_{i1}||G_{11}^{(1i)}|\bold{1}(\Xi^c)\bold{1}(\tilde{\Xi}).$ The remaining terms are estimated in the same manner by using the brutal estimates $|G_{kl}| \leq N$ and $|G_{kl}^{(1i)}|\leq N\,.$
For estimating 
$ \frac{1}{N}\sum\limits_{i=1}^N\mathbb{E}_1|G_{ii}||G_{1i}||H_{i1}||G_{11}^{(1i)}|\bold{1}(\Xi^c)\bold{1}(\tilde{\Xi})$\,, let us define the monomial
$$X_3(H,G)\;:=\; |G_{ii}||G_{1i}||H_{i1}||G_{11}^{(1i)}|\bold{1}(\tilde{\Xi})\,.$$ 
Using Cauchy-Schwarz inequality for the monomial $X_3\,,$ the brutal estimate $|G_{kl}|\leq N$ on $\Xi^c$ and Lemma \ref{Gamma0}\,, we obtain that
\begin{align}
&\frac{1}{N}\sum\limits_{i=1}^N\mathbb{E}_1|G_{ii}||G_{1i}||H_{1i}||G_{11}^{(1i)}|\bold{1}(\Xi^c)\bold{1}(\tilde{\Xi})\nonumber\\ &\leq \frac{1}{N} 
\sum\limits_{i=1}^N\sqrt{\mathbb{P}(\Xi^c)}\sqrt{\mathbb{E}_1|G_{ii}|^2|G_{1i}|^2|H_{1i}|^2|G_{11}^{(1i)}|^2\bold{1}(\tilde{\Xi})}\nonumber\\
&\leq \frac{1}{N}\sum\limits_{i=1}^NN^{-D/2}\sqrt{N^{6}\mathbb{E}_1|H_{1i}|^2}\nonumber\\
&\leq N^{-D/2}N^{5/2}\,.
\end{align}
where in the last inequality we used the estimate for the variance as in the condition $(ii)$ from the definition of the Wigner matrices. 
The final estimate for the four truncated terms from our example is $2N^{11\epsilon/3+5\delta/2-3/2}+3N^{-D/2}N^{5/2}.$
There are four truncated terms for each term of the first category so the final estimate for the terms in the first category
\begin{align}
8N^{11\epsilon/3+5\delta/2-3/2}+12N^{-D/2}N^{5/2}\,.
\end{align}
An example of the truncated terms that appear for a term in the second category is
\begin{align}
&\frac{1}{N}\sum\limits_{i=1}^N\mathbb{E}_1|G_{ii}||H_{i1}||G_{1i}^{(1i)}||G_{1i}||H_{i1}||G_{11}^{(1i)}|(\bold{1}(\Xi)+ \bold{1}(\Xi^c))(\bold{1}(\tilde{\Xi})+\bold{1}(\tilde{\Xi}^c))\nonumber\\
&=\frac{1}{N}\sum\limits_{i=1}^N\mathbb{E}_1|G_{ii}||H_{i1}||G_{1i}^{(1i)}||G_{1i}||H_{i1}||G_{11}^{(1i)}|\bold{1}(\Xi)\bold{1}(\tilde{\Xi})\nonumber\\&+\frac{1}{N}\sum\limits_{i=1}^N\mathbb{E}_1|G_{ii}||H_{i1}||G_{1i}^{(1i)}||G_{1i}||H_{i1}||G_{11}^{(1i)}|\bold{1}(\Xi^c)\bold{1}(\tilde{\Xi})\nonumber\\
&+\frac{1}{N}\sum\limits_{i=1}^N\mathbb{E}_1|G_{ii}||H_{i1}||G_{1i}^{(1i)}||G_{1i}||H_{i1}||G_{11}^{(1i)}|\bold{1}(\Xi)\bold{1}(\tilde{\Xi}^c)\nonumber\\&+\frac{1}{N}\sum\limits_{i=1}^N\mathbb{E}_1|G_{ii}||H_{i1}||G_{1i}^{(1i)}||G_{1i}||H_{i1}||G_{11}^{(1i)}|\bold{1}(\Xi^c)\bold{1}(\tilde{\Xi}^c)\,.
\end{align}
Using the definition of the events $\Xi$ and $\tilde{\Xi}$\,, and Lemma \ref{Gamma0} we obtain directly the estimate
\begin{align}
\frac{1}{N}\sum\limits_{i=1}^N\mathbb{E}_1|G_{ii}||H_{i1}||G_{1i}^{(1i)}||G_{1i}||H_{i1}||G_{11}^{(1i)}|\bold{1}(\Xi)\bold{1}(\tilde{\Xi}) \leq 4N^{10\epsilon/3+4\delta-1}\,.
\end{align}
Proceeding in the same manner as for the terms in the first category, we obtain for the remaining three terms the estimate $3C_4^2N^{3}N^{-D/2}\,,$
where $C_4$ is the constant corresponding to the $L^4$ norm estimate as in the definition of Wigner matrices.
For our particular example of term in the second category, we obtain the final estimate $4N^{10\epsilon/3+4\delta-1}+3C_4^2N^{3}N^{-D/2}\,.$
There are four truncated terms for each term of the second category so the final estimate for the terms in the second category 
\begin{align}
16N^{10\epsilon/3+4\delta-1}+12C_4^2N^{3}N^{-D/2}\,.
\end{align}
We finally obtain that
\begin{align}\label{concentrationdouble}
&-\frac{1}{N}\sum\limits_{i=1}^N\mathbb{E}_1G_{ii}^{(1i)}G_{11}^{(1i)}=-\frac{1}{N}\sum\limits_{i=1}^N\mathbb{E}_1G_{ii}G_{11}\nonumber\\&+ O\left(8N^{11\epsilon/3+5\delta/2-3/2}+12N^{-D/2}N^{5/2}+ 16N^{10\epsilon/3+4\delta-1}+12C_4^2N^{3}N^{-D/2}\right)\,.
\end{align}
Plugging \eqref{concentrationdouble} in the general identity, we obtain that
\begin{align}
&1+z\mathbb{E}_{1}G_{11} = -\frac{1}{N}\sum\limits_{i=1}^N\mathbb{E}_1G_{ii}G_{11}\nonumber\\
&+O\left(8N^{23\epsilon/6+5\delta/2-3/2} + 6N^{-D/2}C_6^3N^{5/2}+ 2N^{15\epsilon/6+3\delta/2-1}+3C_4^2N^{2}N^{-D/2}\right)\nonumber\\
&+O\left(8N^{11\epsilon/3+5\delta/2-3/2}+12N^{-D/2}N^{5/2}+ 16N^{10\epsilon/3+4\delta-1}+12C_4^2N^{3}N^{-D/2}\right)\,.
\end{align}

In order to not carry the estimates, we introduce the following notations
$$
f_1(N, \epsilon, \delta, D)=8N^{23\epsilon/6+5\delta/2-3/2} + 6N^{-D/2}C_6^3N^{5/2}+ 2N^{15\epsilon/6+3\delta/2-1}+3C_4^2N^{2}N^{-D/2},
$$
$$ 
f_2(N, \epsilon, \delta, D)=8N^{11\epsilon/3+5\delta/2-3/2}+12N^{-D/2}N^{5/2}+ 16N^{10\epsilon/3+4\delta-1}+12C_4^2N^{3}N^{-D/2}\,.
$$
Using Lemma \ref{lemma3.1} we obtain that 
\begin{align}
1+z\mathbb{E}_{1}G_{11} &= -\frac{1}{N}\sum\limits_{i=1}^NG_{ii}G_{11}+O_{\prec}\left(\frac{N^{5\delta/2}}{\sqrt{N\eta_0}}\right)\nonumber+O\left( f_1(N, \epsilon, \delta, D)+f_2(N, \epsilon, \delta, D)\right).
\end{align}
Using the Concentration of Measure Lemma \ref{concentration of measure result} and using Lemma $3.4(ii)$ from \cite{benaych2016lectures} \,, we further obtain that
\begin{align}
1+zG_{11}  &= -\frac{1}{N}\sum\limits_{i=1}^NG_{ii}G_{11}+O_{\prec}\left(|z|\frac{N^{3\delta/2}}{\sqrt{N\eta_0}}\right)+O_{\prec}\left(\frac{N^{5\delta/2}}{\sqrt{N\eta_0}}\right)\nonumber\\&+O\left( f_1(N, \epsilon, \delta, D)+f_2(N, \epsilon, \delta, D)\right).
\end{align}
In the same manner, we can redo the computation for $G_{jj} \,,\forall j \in \llbracket 1, N \rrbracket\,.$\\ 
Using that for fixed $ \delta \in (0, \delta_0)\,,$ $N^{5\delta/2} \geq N^{3\delta/2} \,,$ and
using the fact that a deterministic estimate is also an estimate with high probability, we obtain that
\begin{align}
1+zG_{jj} &= -\frac{1}{N}\sum\limits_{i=1}^NG_{ii}G_{jj}+O_{\prec}\left(\frac{(1+|z|)N^{5\delta/2}}{\sqrt{N\eta_0}}\right)\nonumber\\&+O_{\prec}\left( f_1(N, \epsilon, \delta, D)+f_2(N, \epsilon, \delta, D)\right)\,.
\end{align}
Since $\epsilon \in (0,\epsilon_0)$ and large $D>0$ are arbitrary, we can choose them such that all the terms that have $-D$ at the exponent are smaller than the ones without in $f_1(N, \epsilon, \delta, D)$ and in $f_2(N, \epsilon, \delta, D).$ 
Using the definition of $\prec$ to eliminate the constants and writing the terms left in $f_1(N, \epsilon, \delta, D)$ and in $f_2(N, \epsilon, \delta, D)$, we obtain that
\begin{align}
1+zG_{jj} &= -\frac{1}{N}\sum\limits_{i=1}^NG_{ii}G_{jj}+O_{\prec}\left(\frac{(1+|z|)N^{5\delta/2}}{\sqrt{N\eta_0}}\right)\nonumber\\
&+O_{\prec}(N^{23\epsilon/6+5\delta/2-3/2}+ N^{15\epsilon/6+3\delta/2-1}+N^{11\epsilon/3+5\delta/2-3/2} +N^{10\epsilon/3+4\delta-1})\,.
\end{align}
The term $N^{10\epsilon/3+4\delta-1}$ is the biggest term among the four. Estimating for fixed $\delta \in (0, \delta_0)\,,$ $N^{5\delta/2}\leq N^{5\delta}\,,$ using that $\eta_0=N$ and that $|z|>0$ we obtain  
\begin{align}
N^{10\epsilon/3+4\delta-1} \leq (1+|z|)N^{10\epsilon/3+5\delta-1}\,.
\end{align}
Taking the limit $\epsilon \to 0$, we finally obtain that
\begin{align*}
1+zG_{jj} = -\frac{1}{N}\sum\limits_{i=1}^NG_{ii}G_{jj}+O_{\prec}\left(\frac{(1+|z|)N^{5\delta}}{\sqrt{N\eta_0}}\right)\,.
\end{align*}
Averaging over $j \in \llbracket 1, N \rrbracket\,,$ we obtain the conclusion for fixed $\eta_0 >0$\,.
\begin{align}
1+zs+s^2 = O_{\prec}\left(\frac{(1+|z|)N^{5\delta}}{\sqrt{N\eta_0}}\right)\,.
\end{align}
It remains to prove that the equation
\begin{align}\label{3.47}
1+sz+s^2=O_{\prec}\left(\frac{(1+|z|)N^{5\delta}}{\sqrt{N\eta}}\right)
\end{align}
holds for all $ z  \in \bold{S}\,.$ \\
For this, following again the strategy in Lemma $5.5$ in \cite{bauerschmidt2017local}, set
$$ \eta_l=\eta_0+l/N^{4} \,, \hspace{4mm} l \in \llbracket 1, N^5 \rrbracket\,,$$
and $z_l=E+i\eta_l$\,. Since the Concentration of Measure Lemma and Lemma $(3.1)$ hold uniformly for any $\eta \geq \eta_0$\,, using the result for fixed $\eta_0$, and a union bound, we have that \eqref{3.47} holds simultaneously at all $z$ with $l \in \llbracket 1, N^5 \rrbracket$\,.
Since $(\eta_l)_l$ is a $ 1/N^{4}$-net of $\llbracket \eta_0, \eta_0+N \rrbracket$ and $ s$ and $G_{ij}$ are Lipschitz continuous  with the constant $1/\eta^2 \leq N^2 $\,, the claim follows.
\end{proof}

\section{Proof of the Weak Local Law}

In this section, we prove the main result of the paper. The proof uses a bootstrapping technique that is performed for fixed $ \delta \in (0, \delta_0)\,,$ with $\delta_0 \leq \frac{\gamma}{3} \,.$ The main advantage of using the bootstrapping technique is the ability to track information on the bounds on the scale $\eta$ to the scale $\eta/N^{\delta}\,.$ This section is divided into two subsections. The first one contains a stability analysis of the equation $m^2+mz+1=0\,,$ where $m$ is the Stieltjes transform of the Semicircle Law. In the second subsection the main Theorem is proved.
 
\subsection{Stability Analysis}
The Stieltjes transform of the SemiCircle law is the unique solution of the equation
$$m^2+mz+1=0\,. $$

To show that $m$ and $s$ are close we use the stability of the equation $m^2+mz+1=0$ in the form provided by the following deterministic result.
\begin{lemma}[Lemma $5.6$ of \cite{bauerschmidt2017local}]\label{Lem4.1}
Let $s : \mathbb{C}_+ \to \mathbb{C}_+$ be continuous. Set
\begin{align}\label{RR}
R :=s^2+sz+1\,.
\end{align}
For $ E \in \mathbb{R}$\,, $\eta_0 >0 $  and $\eta_1 \geq 3 \vee \eta_0$\,, suppose that there is a non-decreasing continuous function $r: [\eta_0, \eta_1] \to [0,1] $ such that $|R(E+i \eta)| \leq (1+|E+i\eta|)r(\eta)$ for all $\eta \in \llbracket \eta_0, \eta_1 \rrbracket $\,. Then, for all $z=E+i \eta$ with $\eta \in \llbracket \eta_0, \eta_1 \rrbracket$, we have that
$$|s-m|=O(F(r))\,.$$
\end{lemma}

\subsection{Proof of the Weak Local Law using a Bootstrapping Argument}
This subsection contains the proof of the main Theorem of the paper. We first prove a lemma and a proposition that are fundamental for the bootstrapping argument.
The core of the proof of the main Theorem is an induction on the spectral scale, where information about $G$ is passed on from scale $\eta$ to scale $\eta/N^{\delta}\,,$ for $\delta \in (0, \delta_0)\,.$

In order to formalize this, we introduce the random error parameters
\begin{align*}
\Gamma \equiv \Gamma(z) \;:=\; \max_{k,l}|G_{kl}(z)|\vee 1, \hspace{5mm} \Gamma^* \equiv \Gamma^*(z)\;:=\; \sup_{\eta' \geq \eta}\Gamma(E+i\eta')\,.
\end{align*}

We prove the following preliminary result that allows us to propagate the bounds that we obtain for $G$ on a given scale, to weaker bounds on a smaller scale.
\begin{lemma}[Lemma $10.2$ in \cite{benaych2016lectures}]\label{Lem2.1}
For any $M >1$ and $z \in \mathbb{C}_+$\,, we have $\Gamma(E+i\eta/M) \leq M \Gamma(E+i\eta)\,.$
\end{lemma}

The most important ingredient of the proof of the Local Semicircle Law is the following result, similar in spirit with Proposition $2.2$ in \cite{bauerschmidt2017local} \,. However, the proof uses extensively the Resolvent Expansion techniques.
\begin{proposition}\label{Prop4.3}
Let $ \delta \in (0, \delta_0)$ be fixed. If for $ z \in \bold{S}\,,$ we have $ \Gamma^*(z)  \prec N^{\delta}\,,$
it follows that
\begin{align}\label{equiv2.6}
\max_{i \in \llbracket 1, N \rrbracket } |m-G_{ii}| &\prec F\left( \frac{N^{5\delta}}{\sqrt{N\eta}}\right)N^{\delta}\,,\\
\max_{i \neq j \in \llbracket 1, N \rrbracket }|G_{ij}|&\prec \frac{N^{5\delta/2}}{\sqrt{N\eta}}\,.
\end{align}
\end{proposition}
\begin{proof}
Let $z_0=E+i\eta_0 \in \bold{S}\,,$ be given.  Set $\eta_1=N\,.$ Lemma \ref{Lem4.1} shows that with high probability for all $ \eta \in [\eta_0, \eta_1]$  the function $s$ satisfies \eqref{RR} with 
$$|R(z)| \leq (1+|z|)r(\eta)\,,$$
where $r(\eta) := \frac{N^{5\delta}}{\sqrt{N\eta}}$\,. Hence, the function $r$ is decreasing when $\eta$ is increasing. So, we have that $ r \in [0,1]$\,. Using Lemma \ref{Lem4.1} and using the fact that a deterministic estimate is also an estimate with high probability, we obtain that $|m-s|\prec F\left(\frac{N^{5\delta}}{\sqrt{N\eta}}\right)$ for all $ \eta \in [\eta_0, \eta_1]\,.$ \\
We now estimate $G_{jj}-m\,,$ for $j \in \llbracket 1, N \rrbracket\,.$
Before averaging over $j$\,, the estimate
\begin{align}\label{lala}
1+sz+s^2 =O_{\prec}\left(\frac{(1+|z|)N^{5\delta}}{\sqrt{N\eta}}\right)
\end{align}
has the form
\begin{align}\label{lala1}
1+(s+z)G_{jj} =O_{\prec}\left( \frac{(1+|z|)N^{5\delta}}{\sqrt{N\eta}}\right)\,.
\end{align}
Using the estimate in \eqref{lala1} together with $|s-m| \prec F\left(\frac{N^{5\delta}}{\sqrt{N\eta}}\right)\,,$ we find
\begin{equation}
|1+(z+m)G_{jj}| \prec F\left(\frac{N^{5\delta}}{\sqrt{N\eta}}\right)|G_{jj}|+\frac{(1+|z|)N^{5\delta}}{\sqrt{N\eta}}\,.
\end{equation}
Using the equation that the Stieltjes transform satisfies, it is easy to deduce that $z+m=-\frac{1}{m}$\,. Also, using that $m$ is bounded and is the transform of a compactly supported measure, we deduce that $(1+|z|)|m|=O(1)$\,.\\
Hence, we obtain that
\begin{equation}
|m-G_{jj}| \prec F\left(\frac{N^{5\delta}}{\sqrt{N\eta}}\right)|G_{jj}| + \frac{N^{5\delta}}{\sqrt{N\eta}}\,.
\end{equation}
Furthermore, using Lemma $3.4(ii)$ from \cite{benaych2016lectures}, we obtain 
\begin{equation}
|m-G_{jj}| \prec F\left(\frac{N^{5\delta}}{\sqrt{N\eta}}\right)N^{\delta}+ \frac{N^{5\delta}}{\sqrt{N\eta}}\,.
\end{equation}
Therefore, we finally obtain that
\begin{equation}
|m-G_{jj}| \prec F\left(\frac{N^{5\delta/2}}{\sqrt{N\eta}}\right)N^{\delta}\,.
\end{equation}
For estimating the off-diagonal entries we use the resolvent identity.
From $$(HG)_{ij} = ((H-z)G)_{ij}+zG_{ij}=\delta_{ij}+zG_{ij}\,,$$ it follows that 
$$ G_{12}=G_{12}(HG)_{11}-G_{11}(HG)_{12} $$ 
and therefore we have
\begin{align}
\mathbb{E}_1G_{12}=\sum\limits_{i=1}^N\mathbb{E}_1(G_{12}H_{1i}G_{i1}-G_{11}H_{1i}G_{i2})\nonumber\\
\end{align}
Using the resolvent identity for $G_{12}\,,$ $G_{i1}\,,$ $G_{11}$ and $G_{i2}$, we obtain that 
\begin{align}
\mathbb{E}_1\sum\limits_{i=1}^N(G_{12}H_{1i}G_{i1}-G_{11}H_{1i}G_{i2})&=\mathbb{E}_1\sum\limits_{i=1}^N(G_{12}^{(1i)}H_{1i}G_{i1}^{(1i)}-G_{11}^{(1i)}H_{1i}G_{i2}^{(1i)})\nonumber\\
&+\mathbb{E}_1\sum\limits_{i=1}^N|H_{1i}|^2(-G_{12}^{(1i)}G_{ii}G_{11}^{(1i)}-G_{12}^{(1i)}G_{1i}G_{i1}^{(1i)}\nonumber\\
&-G_{i1}^{(1i)}G_{1i}G_{12}^{(1i)}-G_{i1}^{(1i)}G_{11}G_{i2}^{(1i)}\nonumber\\
&+G_{11}^{(1i)}G_{ii}G_{12}^{(1i)}+G_{11}^{(1i)}G_{i1}G_{i2}^{(1i)}\nonumber\\
&+ G_{i2}^{(1i)}G_{1i}G_{11}^{(1i)}+G_{i2}^{(1i)}G_{11}G_{i1}^{(1i)})\,.
\end{align}
We observe that the terms that contain $G_{ii}$ are the same but with opposite signs, so we are left with $6$ terms that contain $|H_{1i}|^2\,.$ Using another resolvent identity for $G_{ii}$, $G_{1i}$, $G_{i1}$ and $G_{11}$, we obtain that
\begin{align}\label{4.21}
&\mathbb{E}_1\sum\limits_{i=1}^N(G_{12}H_{1i}G_{i1}-G_{11}H_{1i}G_{i2})=\mathbb{E}_1\sum\limits_{i=1}^N(G_{12}^{(1i)}H_{1i}G_{i1}^{(1i)}-G_{11}^{(1i)}H_{1i}G_{i2}^{(1i)})\nonumber\\
&+O(\mathbb{E}_1\sum\limits_{i=1}^N|H_{1i}|^2(|G_{12}^{(1i)}||G^{(1i)}_{1i}||G_{i1}^{(1i)}|+|G_{i1}^{(1i)}||G^{(1i)}_{1i}||G_{12}^{(1i)}|+|G_{i1}^{(1i)}||G^{(1i)}_{11}||G_{i2}^{(1i)}|))\nonumber\\
&+O(\mathbb{E}_1\sum\limits_{i=1}^N|H_{1i}|^2(|G_{11}^{(1i)}||G^{(1i)}_{i1}||G_{i2}^{(1i)}|+ |G_{i2}^{(1i)}||G^{(1i)}_{1i}||G_{11}^{(1i)}|+|G_{i2}^{(1i)}||G^{(1i)}_{11}||G_{i1}^{(1i)}|))\nonumber\\
&+O(\mathbb{E}_1\sum\limits_{i=1}^N|H_{1i}|^3\bar{G})\,,
\end{align}
where $\bar{G}$ encodes all the products of entries of $G$ that appear in the corresponding terms.\\
Using the same idea as in \cite{bauerschmidt2017local} , $(5.60)$, we get rid of the terms that contain first powers only of entries of $H\,.$ Hence we are left with terms that contain at least $|H_{1i}|^2\,.$ We split these terms into two categories depending on the numbers of entries of $H$ that they contain. The terms in the same category are estimated in the same manner. We truncate the summations that appear with respect to the events $\Xi$ and $\tilde{\Xi}$ as in \eqref{event}\,. For terms in all categories that are truncated according to at least one low probability event, we use the Cauchy-Schwarz inequality argument as in section $2$\,. Using the definition of stochastic domination we obtain that they are smaller than the remaining terms in \eqref{4.21}\,. We further estimate terms of the form $\mathbb{E}_1|H_{1i}|^2|G_{12}^{(1i)}||G_{1i}||G_{i1}^{(1i)}|\bold{1}(\Xi)\bold{1}(\tilde{\Xi})\,.$
For a term in the first category, using Cauchy-Schwarz inequality, Lemma \ref{Gamma0}\,, Lemma \ref{2.3} and Ward identity for $\eta_0=N$ we obtain that
\begin{align}
\mathbb{E}_1|H_{1i}|^2|G_{12}^{(1i)}||G_{1i}||G_{i1}^{(1i)}|\bold{1}(\Xi)\bold{1}(\tilde{\Xi})&\leq \sqrt{\mathbb{P}(\Xi)}\sqrt{\mathbb{E}_1\sum\limits_{i=1}^NN|G_{12}^{(1i)}|^2|H_{1i}|^4|G_{i1}^{(1i)}|^2|G_{1i}|^2}\nonumber\\
&\leq\frac{4N^{17\epsilon/6+5\delta/2}}{\sqrt{N\eta_0}}\,.
\end{align}
 In addition, since $|H_{ij}|^3 \leq N^{3\epsilon-3/2}$ on $\tilde{\Xi}$, we obtain that the terms in the second category are smaller than the terms in the first category.\\
Using the definition of $\prec\,,$ and taking the limit $\epsilon \to 0\,,$ we obtain that
\begin{align}
|\mathbb{E}_1G_{12}| \prec \frac{N^{5\delta/2}}{\sqrt{N\eta_0}}\,.
\end{align}
Using the same procedure as in the proof of the Concentration of Measure Lemma, we track the bound to smaller scales, i.e.
\begin{align}
|\mathbb{E}_1G_{12}| \prec \frac{N^{5\delta/2}}{\sqrt{N\eta}}\,.
\end{align}
Using the Concentration of Measure Lemma, we obtain that
\begin{align}
|G_{12}-\mathbb{E}_1G_{12}| \prec \frac{N^{3\delta/2}}{\sqrt{N\eta}}
\end{align}
Furthermore using the estimate for $\mathbb{E}_1G_{12}\,,$ and the definition of $ \prec\,,$ we obtain that
\begin{align}
|G_{12}| \prec \frac{N^{3\delta/2}}{\sqrt{N\eta}} +\frac{N^{5\delta/2}}{\sqrt{N\eta}}\,.
\end{align}
We finally obtain that
$$|G_{12}| \prec \frac{N^{5\delta/2}}{\sqrt{N\eta}} \,.$$
By symmetry and a union bound, the claim holds again for $12$ replaced with $ij\,,$ for $i, j \in \llbracket 1, N \rrbracket\,.$
\end{proof}
\noindent

The proof of the Local Law follows with a similar analysis as in \cite{bauerschmidt2017local}\,.

\begin{proof}[\textbf{Proof of the Local Law}]
\noindent
We first note that the proof of the theorem for the cases $|E|>N$ and $\eta \geq N$ is trivial. Hence, it suffices to prove the Theorem for $ z \in \bold{S}\,.$ Since $G$ is Lipschitz continuous in $z$ with Lipschitz constant bounded by $\frac{1}{\eta^2} \leq N^2\,,$ it moreover suffices to prove the Theorem for $z \in \bold{S} \cap (N^{-4}\mathbb{Z}^2)\,.$ By a union bound, it suffices to prove the Theorem for each $E \in [-N,N]\cap N^{-4}\mathbb{Z}\,.$ \\
Fix therefore $E \in [-N, N]\cap (N^{-4}\mathbb{Z})\,.$ \\Let $$K \;:=\;\max\{ k \in \mathbb{N} : N/N^{k\delta} \geq N^{-1+\gamma} \}\,.$$\\
Clearly, $K \leq \left[\frac{2-\gamma}{\delta}\right],$ i.e. $ K \leq \frac{1}{\delta}\,.$ For $k \in \llbracket 0, K \rrbracket\,, $ set $\eta_k=\frac{N}{N^{k\delta}}\,,$ and $z_k:=E+i\eta_k\,.$ By induction on $k$ we shall prove that 
\begin{align}\label{equiv2.7}
\Gamma^*(z_k) \prec 1\,,
\end{align}
for $k \in \llbracket 0 , K \rrbracket\,.$\\
 The claim in \eqref{equiv2.7} is trivial for $k=0\,,$ since then $\eta_k=N$ and therefore we have deterministically $\Gamma^*(z_k) \leq 1\,.$ Now suppose that \eqref{equiv2.7} holds for some $k \in \llbracket 0, K \rrbracket\,.$ Then Lemma \ref{Lem2.1} applied with $\eta=\eta_k$ and $M=N^{\delta}\,, $ implies that 
$\Gamma^*(z_{k+1}) \prec N^{\delta}\,. $
We apply Proposition \eqref{Prop4.3} for $z=z_{k+1}$\,. \\
To this end, we have that $\max_{i \in \llbracket 1, N \rrbracket}|m-G_{ii}| \prec F\left(\frac{N^{5\delta}}{\sqrt{N\eta}}\right)N^{\delta}$ and $\max_{i \neq j \in \llbracket 1, N \rrbracket }|G_{ij}| \prec \frac{N^{5\delta/2}}{\sqrt{N\eta}}\,.$

Since $|m|\leq 1\,,$ $\frac{N^{3\delta/2}}{\sqrt{N\eta}} \leq 1$ and $F\left( \frac{N^{5\delta}}{\sqrt{N\eta}}\right)N^{\delta} \leq 1\,,$ we conclude that $\Gamma^*(z_{k+1}) \leq 1 $\,. Using the fact that a deterministic bound implies high probability bound we finally obtain that $\Gamma^*(z_{k+1}) \prec 1  \,.$
This concludes the proof of the induction step for all $k \in \llbracket 0, K \rrbracket\,.$
After $K$ steps of bootstrapping we have that
\begin{align}
\max_{i \in \llbracket 1, N \rrbracket } |m-G_{ii}|  &\prec F\left( \frac{N^{5\delta}}{\sqrt{N\eta}}\right)N^{\delta}\,,\\
\max_{i \neq j \in \llbracket 1, N \rrbracket }|G_{ij}| &\prec \frac{N^{5\delta/2}}{\sqrt{N\eta}}\,.
\end{align}
Using the definition of $\prec$ we obtain that
\begin{align}
\max_{i \in \llbracket 1, N \rrbracket } |m-G_{ii}| &\prec F\left( \frac{1}{\sqrt{N\eta}}\right)\,,\\
\max_{i \neq j \in \llbracket 1, N \rrbracket }|G_{ij}| &\prec\frac{1}{\sqrt{N\eta}}\,,
\end{align}
which yields the desired conclusion.
\end{proof}



\newpage

\bibliography{WL.bib}

\end{document}